\UseRawInputEncoding
\documentclass[11pt,leqno]{amsart}

\usepackage[top=2.4 cm,bottom=1.9cm,left=2.5 cm,right=2 cm]{geometry}
\usepackage[T1]{fontenc}

\usepackage{bbm}
\usepackage{esint}
\usepackage{graphicx}
\usepackage{color}
\usepackage{amsfonts,amsmath,amssymb,mathrsfs}
\usepackage{psfrag}
\usepackage{enumitem}
\newcounter{stepnb}

\usepackage{hyperref}\hypersetup{colorlinks=true,linkcolor=blue,citecolor=blue}

\newtheorem{theorem}{Theorem}
\newtheorem{lemma}[theorem]{Lemma}
\newtheorem{proposition}[theorem]{Proposition}

\theoremstyle{plain} \newtheorem{corollary}[theorem]{Corollary}

\newtheorem{definition}{Definition}[section]
\theoremstyle{remark}
\newtheorem{remark}[theorem]{Remark}
\allowdisplaybreaks
\theoremstyle{plain} 

\newcommand{\N}{\mathbb{N}}
\newcommand{\Z}{\mathbb{Z}}

\newcommand{\R}{\mathbb{R}}
\newcommand{\Q}{\mathbb{Q}}
\newcommand{\M}{{\mathcal M}}

\newcommand{\D}{{\mathcal D}}
\newcommand{\G}{{\mathcal G}}
\newcommand{\F}{\mathscr{F}}

\newcommand{\Id}{{\mathrm I}}

\renewcommand{\H}{{\mathscr H}}
\renewcommand{\L}{{\mathscr L}}

\renewcommand{\S}{\mathbb{S}}
\renewcommand{\P}{\mathcal{P}}

\DeclareMathOperator{\BV}{BV}

\newcommand{\e}{\varepsilon}

\newcommand{\TV}{\text{\rm TotVar}}
\newcommand{\Lip}{\mathrm{Lip}}

\newcommand{\be}{\begin{equation}}
\newcommand{\eq}{\end{equation}}

\renewcommand{\div}{{\rm div}\,}

\newcommand{\loc}{\mathrm{loc}}

\newcommand{\dist}{\mathrm{dist}}
\newcommand{\diam}{\mathrm{diam}}
\newcommand{\ac}{\mathrm{ac}}
\newcommand{\n}{\mathbf{n}}

\title{Rectifiability of entropy defect measures in a micromagnetics model}
\author[E.~Marconi]{Elio Marconi}
\address{Elio Marconi, EPFL B, Station 8, CH-1015 Lausanne, CH.}
\email{elio.marconi@epfl.ch}
\thanks{The author has been supported by the SNF Grant 182565.}

\begin{document}
\maketitle


\begin{abstract}
We study the fine properties of a class of weak solutions $u$ of the eikonal equation arising as asymptotic domain of a family of energy functionals  introduced in (Rivi\`ere T, Serfaty S. Limiting domain wall energy
for a problem related to micromagnetics. Comm Pure Appl Math 2001; 54(3):294-338).
In particular we prove that the entropy defect measure associated to $u$ is concentrated on a 1-rectifiable set, which detects the jump-type discontinuities of $u$.
\end{abstract}

\section{Introduction}
We consider a bounded simply connected domain $\Omega \subset \R^2$ and we investigate the fine properties of the following class of divergence free unit vector fields:
\begin{definition}\label{D_M_div}
We denote by $\mathcal M_\div (\Omega)$ the set of vector fields $u: \Omega \to \mathbb{C}$ for which the following conditions hold:
\begin{enumerate}
\item $\div u = 0$ in the sense of distributions;
\item there exists $\phi \in L^\infty (\Omega)$ such that $u = e^{i\phi}$ and 
\begin{equation*}
\langle U_\phi, \psi(x,a) \rangle := \int_{\Omega \times \R} e^{i(\phi(x) \wedge a)}\cdot \nabla_x \psi(x,a) dxda \in \M(\Omega\times \R),
\end{equation*}
where $\M(\Omega\times \R)$ denotes the set of finite Radon measures on $\Omega \times \R$.
\end{enumerate}
\end{definition}

The space $\M_\div(\Omega)$ is the conjectured asymptotic domain as $\e\to 0$ of the following family of energy functionals introduced in \cite{RS_magnetism} in the context of micro-magnetics:
\begin{equation*}
E_\e (u) := \e \int_\Omega |\nabla u|^2 + \frac{1}{\e}\int_{\R^2}|H_u|^2,
\end{equation*}
where $u \in W^{1,2}(\Omega,\S^1)$ and the so-called demagnetizing field $H_u \in L^2(\R^2;\R^2)$ is such that
$\mathrm{curl}\, H_u=0$ and $\div (\tilde u + H_u)=0$ in $\mathcal D'(\R^2)$, where 
\begin{equation*}
\tilde u (x) =
\begin{cases}
u (x) & \mbox{if }x \in \Omega ; \\
0 & \mbox{otherwise}.
\end{cases}
\end{equation*}
The following compactness result was proven in \cite{RS_magnetism2}: let $\phi_{\e_n}$ be a bounded sequence in $L^\infty(\Omega)$ such that $E_{\e_n} (u_{\e_n})$ is uniformly bounded, where $u_{\e_n}=e^{i\phi_{\e_n}}$ and $\e_n\to 0$;
then $\phi_{\e_n}$ is relatively compact in $L^p(\Omega)$ for every $p\in [1,\infty)$ and for every limit point $\bar \phi$ it holds
\begin{equation}\label{E_liminf}
e^{i\bar \phi}\in \M_\div(\Omega)\qquad \mbox{and} \qquad |U_{\bar \phi}|(\Omega \times \R)\le \liminf_{n\to \infty}E_{\e_n}(u_{\e_n}).
\end{equation}
Although the $\Gamma$-$\liminf$ inequality \eqref{E_liminf} was proved in full generality, the corresponding $\Gamma$-$\limsup$ inequality was obtained only in special cases. In particular the energy-minimizing configurations were characterized by the results in \cite{RS_magnetism2,ALR_viscosity}.
It is expected that the energy $E_\e$ is concentrated on lines at a scale $\e>0$ around the lines, allowing for sharper and sharper jumps as $\e\to 0$; the latters correspond in three dimensions to jumps across surfaces, called \emph{domain walls} in the theory of micromagnetism (see \cite{RS_magnetism}).
These lines are detected by the measure $U_\phi$: in particular if we denote by $p_x:\Omega\times \R$ the standard projection on the first component and if $\phi \in \BV(\Omega)$, then the measure 
\begin{equation*}
\nu:= (p_x)_\sharp |U_\phi|
\end{equation*}
is concentrated on the 1-rectifiable jump set of $\phi$.

However, vector fields in $\M_\div(\Omega)$ do not have necessarily bounded variation and a study of their fine properties must therefore
be independent of the theory of $\BV$ functions. This program was announced in \cite{ALR_viscosity} and carried on in \cite{AKLR_rectifiability} leading to the following result:

\begin{theorem}\label{T_intro}
Let $\phi$ be a lifting of $u\in \M_\div(\Omega)$ as in Definition \ref{D_M_div}. Then 
\begin{enumerate}
\item The jump set $J$ of $\phi$ is countably $\H^1$-rectifiable and coincides, up to $\H^1$-negligible sets, with
\begin{equation}\label{E_Sigma}
\Sigma:= \left\{ x \in \Omega: \limsup_{r\to 0}\frac{\nu(B_r(x))}{r}>0 \right\}.
\end{equation}
Moreover for every $a\in \R$ it holds
\begin{equation}\label{E_formula}
\div (e^{i\phi\wedge a})\llcorner J = \mathbf{1}_{\phi^-<a<\phi^+}\left(e^{ia}- e^{i\phi^-}\right)\cdot \n_J \H^1\llcorner J,
\end{equation}
where $\n_J$ denotes the normal to $J$.
\item Every $x \in \Omega \setminus \Sigma$ is a vanishing mean oscillation point of $\phi$, namely
\begin{equation*}
\lim_{r \to 0}\frac{1}{r^2}\int_{B_r(x)}|\phi- \phi_r(x)| = 0,
\end{equation*}
where $\phi_r(x)$ is the average of $\phi$ on $B_r(x)$.
\item The measure $\nu \llcorner (\Omega \setminus J)$ is orthogonal to $\H^1$, namely
\begin{equation*}
B \subset (\Omega \setminus J) \mbox{ Borel with }\H^1(B)< \infty \quad \Longrightarrow \quad \nu(B)=0.
\end{equation*}
\end{enumerate}
\end{theorem}

We observe that for functions $\phi\in \BV_\loc(\Omega)$ the above properties (2) and (3) can be improved to
\begin{enumerate}
\item[(2')] $\H^1$-a.e. point in $\Omega\setminus J$ is a Lebesgue point of $\phi$;
\item[(3')] the measure $\nu\llcorner (\Omega \setminus J)$ is identically 0.
\end{enumerate} 
In \cite{AKLR_rectifiability} it was conjectured that both (2') and (3') hold for every $u\in \M_\div(\Omega)$.
The following weaker version of (2') was recently obtained  in \cite{LO_Burgers} in the close setting of weak solutions $u$ with finite entropy production of the Burgers equation:
\begin{enumerate}
\item[(2*)] the set of non Lebesgue point of $u$ has Hausdorff dimension at most 1.
\end{enumerate}
This result was extended for general conservation laws in \cite{M_Lebesgue}, implying in particular that Property (2*) holds in the setting of this paper, namely for functions $\phi \in L^\infty$ corresponding to vector fields $u\in \M_\div(\Omega)$.

The main result of this paper is the proof of property (3') for general vector fields $u \in \M_\div(\Omega)$.

\begin{theorem}\label{T_main}
Let $\phi$ be a lifting of $u\in \M_\div(\Omega)$ as in Definition \ref{D_M_div}. Then the measure $\nu$ is concentrated on 
the countably $\H^1$-rectifiable set $\Sigma$ defined in \eqref{E_Sigma}. In particular for every $a\in \R$ it holds
\begin{equation*}
\div (e^{i\phi\wedge a}) = \mathbf{1}_{\phi^-<a<\phi^+}\left(e^{ia}- e^{i\phi^-}\right)\cdot \n_J \H^1\llcorner J.
\end{equation*}
\end{theorem}
Theorem \ref{T_main} establishes that the concentration property expected for the $\Gamma$-limit functional of $E_\e$ as $\e\to 0$ holds for the candidate $\Gamma$-limit; this property is also considered as a fundamental step to complete the $\Gamma$-$\limsup$ analysis (see \cite{Lecumberry_magnetic}).

\subsection{Main tool and strategy of the proof}
The strategy of the proof of Theorem \ref{T_main} was introduced in \cite{M_Burgers} to prove the analogous result for weak solutions with finite entropy 
production of Burgers equation (or more in general 1d scalar conservation laws with uniformly convex flux). Indeed there is a strong analogy between weak solutions to conservation laws with finite entropy 
production and the solutions to the eikonal equation arising in this model or the related model introduced by Aviles and Giga in \cite{AG_conjecture}. 
In particular Theorem \ref{T_intro} has an analogous version for scalar conservation laws (see \cite{Lecumberry_PhD,DLW_structure}) and for the model by 
Aviles and Giga \cite{DLO_JEMS}. In order to compare the setting of this paper and the one of conservation laws we observe that for $u=e^{i\phi}\in \M_\div(\Omega)$ it holds
\begin{equation*}
\partial_{x_1}cos \phi + \partial_{x_2}\sin \phi=0.
\end{equation*}
Let us assume that $\phi$ takes values in $(0,\pi)$ so that the cosine is invertible in the range of $\phi$ and $v=\cos \phi$ satisfies the equation
\begin{equation*}
\partial_{x_1}v + \partial_{x_2}(\sin(\cos^{-1}(v)))=0.
\end{equation*}
Being the map $\sin \circ \cos^{-1}$ convex on $(-1,1)$, it is possible to transfer the results obtained for conservation laws with convex 
fluxes to solutions of the eikonal equation taking values in $(0,\pi)$. 
When instead the oscillation of $\phi$ is larger than $\pi$, the approach above fails and more refined arguments are needed.
 
The main tool used to prove Theorem \ref{T_main} is the so called Lagrangian representation, which was introduced in \cite{BBM_multid} for entropy
solutions to general conservation laws and then extended in \cite{M_Lebesgue} to weak solutions with finite entropy production.
This Lagrangian representation (see Definition \ref{D_Lagr}) is an extension of the classical method of characteristics to this non-smooth setting and it is strongly inspired
by the Ambrosio's superposition principle in the context of positive measure valued solutions to the linear continuity equation.
Roughly speaking the evolution of the solution is obtained as superposition of single trajectories traveling with characteristic speed.
This tool is well suited for our purposes since also the kinetic measure $U_\phi$ can be decomposed along the characteristic trajectories detected by the Lagrangian representation.
In Section \ref{S_Lagrangian} we prove the existence of a Lagrangian representation for vector fields in $\M_\div(\Omega)$ building on the following kinetic formulation obtained in \cite{RS_magnetism2} (see also \cite{JP_kinetic} in the study of the model by Aviles and Giga and the fundamental paper \cite{LPT_kinetic} in the setting of entropy solutions to scalar conservation laws):
setting $\chi(x,a):=\mathbf{1}_{\phi(x)\ge a}$ it holds
\begin{equation}\label{E_kinetic}
i e^{ia}\cdot \nabla_x \chi = -\partial_a U_\phi \qquad \mbox{in }\mathcal D'(\Omega \times \R).
\end{equation}
The proof of the existence of a Lagrangian representation follows the strategy of \cite{M_Lebesgue}, but additional work is required since we consider here solutions on bounded domains instead of the whole $\R^2$.

Once a Lagrangian representation is available for vector fields in $\M_\div(\Omega)$, we implement the strategy introduced in \cite{M_Burgers} to prove Theorem \ref{T_main}. Being the oscillation of $\phi$ bigger than $\pi$ the argument does not apply straightforwardly. 
Still a partial result is obtained in Section \ref{Ss_nu_l} by covering the image of $\phi$ with finitely many intervals $(I_l)_{l=1}^L$ of length less than $\pi$ and appropriately localizing the argument of \cite{M_Burgers}.
A new regularity estimate is proven in Section \ref{Ss_nu_j} and this allows to conclude the proof of Theorem \ref{T_main}, relying on Theorem
\ref{T_intro}.

%

\section{Preliminaries}

\subsection{Duality for $L^1$-optimal transport}
In this section we recall a few facts about $L^1$-optimal transport. We state the results in the form that we will need in Section \ref{S_Lagrangian}. 
\begin{definition}
Let $(X,d)$ be a complete and separable metric space and let $\mu_1,\mu_2 \in \M_+(X)$ be such that $\mu_1(X)=\mu_2(X)$. The Wasserstein distance of order 1 between
$\mu_1$ and $\mu_2$ is defined by
\begin{equation}\label{E_W1}
W_1(\mu_1,\mu_2):=\inf_{\pi\in \Pi(\mu_1,\mu_2)}\int_Xd(x,y)d\pi(x,y),
\end{equation}
where $\Pi(\mu_1,\mu_2)$ is the set of transport plans from $\mu_1$ to $\mu_2$, i.e.
\begin{equation*}
\Pi(\mu_1,\mu_2):=\{\omega\in \mathcal \M_+(X^2): {\pi_1}_\sharp \omega = \mu_1, {\pi_2}_\sharp \omega = \mu_2 \},
\end{equation*}
denoting by $\pi_1,\pi_2:X^2 \to X$ the two natural projections.
\end{definition}

Notice that $W_1$ can take value $+\infty$.

In order to prove the existence of a Lagrangian representation for vector fields in $\M_\div(\Omega)$ we will take advantage of the dual formulation of the $L^1$-optimal transport. The following duality formula can be found for example in \cite{V_oldnew}.
\begin{proposition}\label{P_duality}
For any $\mu_1,\mu_2 \in \M_+(X)$ with $\mu_1(X)=\mu_2(X)$, it holds
\begin{equation*}
W_1(\mu_1,\mu_2)= \sup_{ \psi \in L^1(\mu_1), \|\psi\|_\Lip \le 1} \left( \int_X\psi d\mu_1-\int_X \psi d\mu_2 \right).
\end{equation*}
\end{proposition}

Since it will be convenient to allow that the two measures $\mu_1,\mu_2$ have different masses,
we deduce from Proposition \ref{P_duality} the following result.
\begin{corollary}\label{C_different_masses}
Let $(X,d)$ be bounded and let $\mu_1,\mu_2\in \M_+(X)$. Assume that there exist $C_1,C_2>0$ such that for every $\psi \in \Lip(X)$ it holds
\begin{equation*}
\left| \int_X \psi d\mu_1 - \int_X \psi d \mu_2 \right| \le C_1 |\psi|_\Lip + C_2 \|\psi\|_{L^\infty}.
\end{equation*}
Then there exist $\tilde \mu_1\le \mu_1, \tilde \mu_2\le \mu_2$ such that 
$\|\mu_1-\tilde \mu_1\|\le C_2, \|\mu_2-\tilde \mu_2\|\le C_2$ and 
\begin{equation*}
W_1(\tilde \mu_1,\tilde \mu_2)\le C_1 + C_2 \diam(X).
\end{equation*}
\end{corollary}
\begin{proof}
We assume without loss of generality that $\alpha:= \|\mu_1\|-\|\mu_2\|>0$. Let $\bar \mu_2= \mu_2 + \alpha \delta_{\bar x}$ for some $\bar x \in X$. Then we have
\begin{equation*}
\begin{split}
\left| \int_X \psi d \mu_1 - \int_X \psi d\bar \mu_2\right| = &~ \left| \int_X (\psi - \psi(\bar x))d\mu_1 -  \int_X (\psi - \psi(\bar x))d\bar \mu_2 \right|\\
= &~  \left| \int_X (\psi - \psi(\bar x))d\mu_1 -  \int_X (\psi - \psi(\bar x))d\bar \mu_2 \right| \\
\le &~ C_1|\psi|_\Lip + C_2 |\psi|_\Lip \diam(X).
\end{split}
\end{equation*}
By Proposition \ref{P_duality} it follows that $W_1(\mu_1,\bar \mu_2)\le C_1 + C_2 \diam(X)$. 
Let $\pi \in \mathcal M(X^2)$ be an optimal plan with marginals $\mu_1$ and $\bar \mu_2$ and let $\tilde \pi \le \pi$ be such that $(p_2)_\sharp \tilde \pi = \mu_2$. Then the statement is true for $\tilde \mu_1 = (p_1)_\sharp \tilde \pi$ and $\tilde \mu_2 = \mu_2$.
\end{proof}
 
The next theorem from \cite{BD_L1map} provides the existence of an $L^1$-optimal map with respect to quite general distances on $\R^N$.
\begin{theorem}\label{T_opt_map}
Let $X=\R^N$ with $N\in\N$ be the Euclidean space equipped with the distance induced by a convex norm $|\cdot|_{D*}$. Let $\mu_1,\mu_2\in \mathcal P(\R^N)$ be two probability measures such that 
$\mu_1\ll \mathscr L^N$ and the infimum in \eqref{E_W1} is finite. Then there exists an optimal plan $\pi$ in \eqref{E_W1} induced by a map, i.e. there exists a measurable map $T:\R^N\to \R^N$
such that $T_\sharp \mu_1=\mu_2$ and
\begin{equation*}
W_1(\mu_1,\mu_2) = \int_X|T(x)-x|_{D*}d\mu_1(x).
\end{equation*}
\end{theorem}

\subsection{Weak convergence of measures}
Given a metric space $X$, we denote by $\M_+(X)$ the set of finite non-negative Borel measures on $X$.
We will say that a sequence of measures $(\mu_n)_{n\in \N} \subset \M_+(X)$ is \emph{narrowly convergent} to $\mu \in \M_+(X)$ if
\begin{equation*}
\lim_{n\to \infty }\int_X f d\mu_n = \int_X f d \mu, \qquad \forall f \in   C_b(X),
\end{equation*} 
where $C_b(X)$ denotes the set of continuous real valued bounded functions on $X$.
We moreover say that a bounded family $\F\subset \M_+(X)$ is \emph{tight} if for every $\e>0$ there exists a compact set $K\subset X$ such that for every $\mu \in \F$ it holds
\begin{equation*}
\mu(X\setminus K)<\e.
\end{equation*}
The following classical theorem characterizes the relatively compact families in $\M_+(X)$ (see \cite{Billingsley}).
\begin{theorem}[Prokhorov]\label{T_Prok}
Let $X$ be a metric space. If a bounded family $\F\subset \M_+(X)$ is tight, then it is relatively compact with respect to the narrow convergence. If moreover $X$ is complete and separable then also the converse implication holds.
\end{theorem}

%
%

\section{Lagrangian representation for vector fields in $\M_\div$.}\label{S_Lagrangian}
In this section we introduce the notions of Lagrangian representations of the hypograph and of the epigraph for
the liftings $\phi$ of vector fields in $\M_\div(\Omega)$. We moreover provide a suitable decomposition along characteristics of
the kinetic measure $U_\phi$ introduced in \eqref{E_kinetic}.

\subsection{Notation and main definition}
We will consider the standard decomposition of the measure $Df \in \M(\R)$, where $f\in \BV(\R,\R)$ (see for example \cite{AFP_book}).
We will adopt the following notation:
\begin{equation*}
Df= D^{\ac}f + D^c f + D^jf = \tilde Df + D^jf,
\end{equation*}
where $D^{\ac}f$, $D^cf$ and $D^jf$ denote the absolutely continuous part, the Cantor part and the atomic part of $Df$ respectively; we refer to $\tilde Df$ as the diffuse part of $Df$.

For every function $\phi:\Omega\to [0,M]$ we denote its hypograph and its epigraph by
\begin{equation*}
H_\phi:= \{(x,a)\in \Omega\times [0,M]: a \le \phi(x)\} \qquad \mbox{and} \qquad E_\phi:= \{(x,a)\in \Omega\times [0,M]: a\ge \phi(x)\}
\end{equation*}
respectively.

We denote by $B_R$ an open ball of radius $R$ such that $\overline{B_R}\subset \Omega$ and we set
\begin{equation*}
 \Gamma:= \left\{ (\gamma,t^-_\gamma,t^+_\gamma) : 0\le t^-_\gamma\le t^+_\gamma\le 1, \gamma\in \BV \big((t^-_\gamma,t^+_\gamma); B_R\times [0,M]\big), 
 \gamma_x \mbox{ is Lipschitz} \right\}.
\end{equation*}
For every $t \in (0,1)$ we consider the section 
\begin{equation*}
 \Gamma(t):= \{(\gamma,t^-_\gamma,t^+_\gamma)\in  \Gamma: t \in (t^-_\gamma,t^+_\gamma)\}.
\end{equation*}
%
and we denote by 
\begin{equation*}
\begin{split}
 e_t: \Gamma(t) &\to B_R \times [0,M] \\
(\gamma, t^-_\gamma, t^+_\gamma) & \mapsto  \gamma(t).
\end{split}
\end{equation*}
Sometimes we will identify the triple $(\gamma,t^-_\gamma,t^+_\gamma) \in  \Gamma$ with the curve $\gamma$ itself to make the notation less heavy. 

\begin{definition}\label{D_Lagr}
Let $u \in \M_\div(\Omega)$ and $\phi \in L^\infty(\Omega)$ as in Definition \ref{D_M_div}. We say that the Radon measure $\omega_h \in \M( \Gamma)$ is a \emph{Lagrangian representation} of the hypograph of $\phi$ on $B_R$ if the following conditions hold:
\begin{enumerate}
\item for every $t\in (0,1)$ it holds
\begin{equation}\label{E_repr_formula}
( e_t)_\sharp \left[ \omega_h \llcorner  \Gamma(t)\right]= \mathscr L^{3}\llcorner H_{\phi};
\end{equation}
\item the measure $\omega_h$ is concentrated on the set of curves $\gamma\in  \Gamma$ such that for $\L^1$-a.e. $t \in (t^-_\gamma,t^+_\gamma)$ 
the following characteristic equation holds:
\begin{equation}\label{E_characteristic}
\dot\gamma_x(t)= i e^{i \gamma_a(t)};
\end{equation}
\item it holds the integral bound
\begin{equation}\label{E_reg}
\int_{ \Gamma} \TV_{[0,1)} \gamma_a d\omega_h(\gamma) <\infty.
\end{equation}
\end{enumerate}
Similarly we say that $\omega_e \in \M( \Gamma)$ is a \emph{Lagrangian representation} of the epigraph of $u$ on $B_R$ if Conditions (2) and (3) hold and (1) is replaced by
\begin{equation}\label{E_repr_e}
( e_t)_\sharp \left[ \omega_e \llcorner  \Gamma(t)\right]= \mathscr L^{3}\llcorner E_{\phi} \qquad \mbox{for every }t \in (0,1).
\end{equation}
\end{definition}
In the following we will adopt the slight abuse of notation
\begin{equation*}
( e_t)_\sharp \omega_h := ( e_t)_\sharp \left( \omega_h\llcorner  \Gamma(t) \right).
\end{equation*}

A fundamental property of the Lagrangian representations $\omega_h,\omega_e$ above is that it is possible to decompose the Radon measure $U_\phi$ along the characteristic curves.
 
Given $\gamma \in \Gamma$ we consider
\begin{equation*}
\mu_\gamma=(\Id, \gamma)_\sharp  \tilde D_t \gamma_a + \H^1\llcorner E_{\gamma}^+ -\H^1\llcorner E_{\gamma}^- \in \M((0,1)\times B_R \times [0,M]),
\end{equation*}
where
\begin{equation*}
\begin{split}
E_\gamma^+:=&\{(t,x,a): \gamma_x(t)=x, \gamma_a(t-)<\gamma_a(t+), a \in (\gamma_a(t-),\gamma_a(t+)) \}, \\
E_\gamma^-:=&\{(t,x,a): \gamma_x(t)=x, \gamma_a(t+)<\gamma_a(t-), a \in (\gamma_a(t+),\gamma_a(t-)) \},
\end{split}
\end{equation*}
$\Id:[0,1)\to [0,1)$ denotes the identity and $\tilde D_t \gamma_a$ denotes the diffuse part of the measure $D_t\gamma_a$.

The main result of this section is the following theorem.

\begin{theorem}\label{T_Lagrangian}
Let $u \in \M_\div(\Omega)$ and $\phi \in L^\infty(\Omega)$ as in Definition \ref{D_M_div}.
Let $B_R$ be an open ball of radius $R$ such that $\overline{B_R}\subset \Omega$ and $\H^1$-a.e. $x\in \partial B_R$ is a Lebesgue point of $\phi$.
Then there exist $\omega_h, \omega_e$ Lagrangian representations of the hypograph and of the epigraph of $u$ respectively on $B_R$ enjoying the additional properties:
\begin{equation}\label{E_dec_1}
\int_{ \Gamma}\mu_\gamma d\omega_h(\gamma)=\L^1 \times U_\phi = -\int_{ \Gamma}\mu_\gamma d\omega_e(\gamma),
\end{equation}
\begin{equation}\label{E_dec_2}
\int_{\Gamma}|\mu_\gamma| d\omega_h(\gamma)=\L^1 \times| U_\phi | = \int_{\Gamma}|\mu_\gamma| d\omega_e(\gamma).
\end{equation}
\end{theorem}
The equations \eqref{E_dec_1} and \eqref{E_dec_2} are equalities in the space $\M((0,1)\times B_R \times [0,M])$; 
Eq. \eqref{E_dec_1} asserts that the measure $\L^1\times U_\phi$ can be decomposed along characteristics and Eq. \eqref{E_dec_2} says that it can be done minimizing
\begin{equation*}
\int_{\Gamma} \TV_{(0,1)} \gamma_a d\omega_h(\gamma) \qquad \mbox{and} \qquad \int_{\Gamma} \TV_{(0,1)} \gamma_a d\omega_e(\gamma). 
\end{equation*}
Moreover it follows from \eqref{E_dec_1} and \eqref{E_dec_2} that we can separately represent the negative and the positive parts of $\L^1\times U_\phi$ in terms of the negative and positive parts of the measures $\mu_\gamma$:
\begin{equation}\label{E_mu-}
\int_{\Gamma}\mu^-_\gamma d\omega_h(\gamma)=\L^1\times U_\phi^- = \int_{\Gamma}\mu^+_\gamma d\omega_e(\gamma) \qquad \mbox{and} \qquad 
\int_{\Gamma}\mu^+_\gamma d\omega_h(\gamma)=\L^1\times U_\phi^+ = \int_{\Gamma}\mu^-_\gamma d\omega_e(\gamma).
\end{equation} 

The proof of Theorem \ref{T_Lagrangian} follows the strategy used in \cite{M_Lebesgue} to deal with general conservation laws; some additional work is required to obtain representation of solutions defined on $B_R$ and not on the whole Euclidean space. 

\subsection{An $L^1$-transport estimate}
In this section we prove an $L^1$-transport estimate that will be used as building block in the construction of approximate characteristics. 
We first need the following lemma.

\begin{lemma}\label{L_chi12}
Let $B_R$ be an open ball of radius $R$ such that $\overline{B_R}\subset \Omega$ and $\H^1$-a.e. $x\in \partial B_R$ is a Lebesgue point of $\phi$. Let $\bar t>0$ be such that $\bar t < \dist(B_R,\partial \Omega)$ and let $\chi,U_\phi$ be as in \eqref{E_kinetic}.
We define $\chi^1,\chi^2: [0,\bar t]\times \Omega\times [0,M] \to \{0,1\}$ as
\begin{equation*}
\chi^1(t,x,a)= \chi(x,a) \mathbf{1}_{B_R}(x) \qquad \mbox{and} \qquad \chi^2(t,x,a)=\chi(x-ie^{ia}t,a)\mathbf{1}_{B_R}(x).
\end{equation*}
Then there exist two Radon measure $\mu^1_{\bar t},\mu^2_{\bar t}\in \M([0,\bar t]\times \Omega\times [0,M])$ absolutely continuous with respect to $\H^3\llcorner( [0,\bar t]\times \partial B_R \times [0,M])$ such that 
\begin{equation}\label{E_chi12}
\begin{split}
\partial_t \chi^1+ ie^{ia}\cdot \nabla_x \chi^1 = &~ -\partial_a (\mathbf{1}_{B_R\times [0,M]}U_\phi) + \mu^1_{\bar t},\\
\partial_t \chi^2 + i e^{ia}\cdot \nabla_x \chi^2 = &~ \mu^2_{\bar t},  \\
\e_{\bar t}:= \frac{\|\mu^1_{\bar t} - \mu^2_{\bar t} \|}{\bar t} \to &~ 0 \qquad \mbox{as }\bar t \to 0.
\end{split}
\end{equation}
\end{lemma}
\begin{proof}
By Theorem \ref{T_intro} it holds $\nu(\partial B_R)=0$ so that the first equation in \eqref{E_chi12} holds for 
\begin{equation*}
\mu^1_{\bar t}= g \H^3\llcorner \left( [0,\bar t]\times \partial B_R \times [0,M]\right) \qquad \mbox{with} \qquad g(t,x,a)= ie^{ia}\cdot n(x) \chi(x,a),
\end{equation*}
where $n$ denotes the inner normal to $B_R$ and the dot denotes the scalar product of vectors in $\R^2$.
From definition of $\chi^2$, the second equation in \eqref{E_chi12} holds with 
\begin{equation*}
\mu^2_{\bar t}= i e^{ia}\cdot n(x)  \chi(x-ie^{ia}t,a) \H^3\llcorner  \left( [0,\bar t]\times \partial B_R \times [0,M]\right).
\end{equation*}
In particular 
\begin{equation*}
\|\mu^1_{\bar t}-\mu^2_{\bar t}\| = \int_0^{\bar t}\int_0^M\int_{\partial B_R}|\chi(x,a) - \chi (x-ie^{ia}t,a)| d\H^1(x) da dt = o(\bar t) \qquad \mbox{as }\bar t \to 0
\end{equation*}
since $\H^1$-a.e. $x\in \partial B_R$ is a Lebesgue point of $\phi$ and therefore $\H^2$-a.e. $(x,a) \in \partial B_R \times [0,M]$ is a Lebesgue point of $\chi$.
\end{proof}
%
%
%
%

\begin{proposition}\label{P_L1_estimate}
In the setting of Lemma \ref{L_chi12}, let $\psi \in C^1_c(\Omega \times \R)$. Then
\begin{equation*}
\int_{\Omega \times \R}\psi(x,a)(\chi^1(\bar t)-\chi^2(\bar t))dx da \le \left( \bar t \| \partial_a \psi \|_{L^\infty} + \frac{\bar t^2}{2}\|\nabla_x \psi\|_{L^\infty}\right)\nu ( B_R) + \|\psi\|_{L^\infty}\e_{\bar t} \bar t.
\end{equation*} 
\end{proposition}

\begin{proof}
We set $\tilde \chi := \chi^1-\chi^2$ and $\tilde \psi(t,x,a):=\psi(x+ie^{ia}(\bar t - t),a)$. 
It is straightforward to check that
\begin{equation}\label{E_kin_test}
\partial_t(\tilde\chi\tilde \psi)+ie^{ia}\cdot\nabla_x(\tilde \chi \tilde \psi)=-\tilde \psi \partial_a(\L^1 \times U_\phi) + \tilde \psi (\mu^1_{\bar t}-\mu^2_{\bar t}) \qquad \mbox{in }\D'((0,\bar t) \times \Omega \times \R ).
\end{equation}
Let $g:[0,\bar t]\to \R$ be defined by
\begin{equation*}
g(t) = \int_{\Omega\times \R}\tilde \chi(t) \tilde \psi(t) dxda.
\end{equation*}
It follows from \eqref{E_kin_test} that
\begin{equation*}
g'(t)=-\int_{\Omega\times \R }\partial_a\tilde \psi(t)d U_\phi + \int_{\Omega \times \R} \tilde \psi(t) d (\mu^1_{\bar t} - \mu^2_{\bar t})_t
\end{equation*}
holds in the sense of distributions, where $(\mu^1_{\bar t}-\mu^2_{\bar t})_t$ denotes the disintegration of the measure $\mu^1_{\bar t}-\mu^2_{\bar t}$ in $t\in (0,\bar t)$ with respect to $\L^1\llcorner (0,\bar t)$.
Therefore $g\in C^1([0,\bar t])$ and since $g(0)=0$ it holds
\begin{equation*}
\begin{split}
\int_{\Omega \times \R}\psi(\chi^1(\bar t)-\chi^2(\bar t))dx da =  &~ g(\bar t)-g(0) \\
= &~ \int_0^{\bar t}g'(t)dt \\
=&~ -\int_0^{\bar t}\int_{\Omega \times \R} \partial_a\tilde \psi(t)d U_\phi dt  + \int_{(0,\bar t)\times \Omega \times \R} \tilde \psi d (\mu^1_{\bar t} - \mu^2_{\bar t})\\
=&~ - \int_0^{\bar t}\int_{\Omega \times \R} \left(\partial_v\phi -(\bar t - t)e^{ia}\cdot \nabla_x \psi\right)d U_\phi dt + \int_{(0,\bar t)\times \Omega \times \R} \tilde \psi d (\mu^1_{\bar t} - \mu^2_{\bar t}) \\
\le &~ \left(\bar t \|\partial_a\psi\|_{L^\infty} + \frac{\bar t^2}{2}\|\nabla_x \psi\|_{L^\infty}\right)\nu(B_{R}) 
+ \|\psi\|_{L^\infty} \|\mu^1_{\bar t}-\mu^2_{\bar t}\|
\end{split}
\end{equation*}
and this concludes the proof.
\end{proof}

We set $L_{\bar t}=(\e_{\bar t}\vee \bar t)^{-\frac{1}{2}}$ and we consider the anisotropic distance
\begin{equation*}
\begin{split}
d_{\bar t} : (B_R\times [0,M])^2 & \to [0,+\infty) \\
((x_1,a_1),(x_2,a_2)) & \mapsto L_{\bar t}|x_1-x_2| + |a_1-a_2|.
\end{split}
\end{equation*}
A test function $\psi: B_R\times [0,M] \to \R$ is 1-Lipschitz with respect to $d_{\bar t}$ if and only if
\begin{equation*}
\|\partial_a \psi\|_{L^\infty}\le 1 \qquad \mbox{and} \qquad \|\nabla_x\psi\|_{L^\infty}\le L_{\bar t}.
\end{equation*}
Applying Corollary \ref{C_different_masses} to $\mu^1= \chi^1(\bar t) \L^3$, $\mu^2=\chi^2(\bar t) \L^3$ on the space $(B_R\times [0,M], d_{\bar t})$ we obtain the following result as a consequence of Proposition \ref{P_L1_estimate} and Theorem \ref{T_opt_map}.

\begin{corollary}\label{C_T}
There exists $\rho^1_{\bar t}\le \chi^1(\bar t)$ and $\rho^2_{\bar t}\le \chi^2(\bar t)$ such that
\begin{equation*}
\int_{B_R\times [0,M]} (\chi^1(\bar t) - \rho^1_{\bar t})dxda \le \e_{\bar t}\bar t, \qquad 
\int_{B_R\times [0,M]} (\chi^2(\bar t) - \rho^2_{\bar t})dxda \le \e_{\bar t}\bar t
\end{equation*}
and 
\begin{equation*}
W_1(\rho^1_{\bar t}\L^3, \rho^2_{\bar t}\L^3) \le \left(\bar t + \bar t^{\frac{3}{2}}\right)\nu(B_R) + \e_{\bar t}^{\frac{1}{2}}\bar t \left( 2R + \e_{\bar t}^{\frac{1}{2}}M\right).
\end{equation*}
In particular there exists $T=(T_x,T_a):B_R\times [0,M] \to B_R\times [0,M]$ such that 
$T_\sharp \left(\rho^2_{\bar t}\L^3\right) = \rho^1_{\bar t}\L^3$ and
\begin{equation*}
\int_{B_R\times [0,M]}\left(L_{\bar t} |T_x(x,a)-x| + |T_a(x,a)-a|\right)\rho^2_{\bar t}(x,a)dxda \le 
\left(\bar t + \bar t^{\frac{3}{2}}\right)\nu(B_R) + \e_{\bar t}^{\frac{1}{2}}\bar t \left( 2R + \e_{\bar t}^{\frac{1}{2}}M\right).
\end{equation*}
\end{corollary}

\subsection{Construction of approximate characteristics}
\subsubsection{Building block}
For a fixed $\bar t>0$ we consider the following sets:
\begin{equation*}
\begin{split}
E_1:= & \{(x,a) \in B_R\times [0,M]: x+ie^{ia}\bar t \in B_R\}; \\
E_2:= & \{(x,a) \in B_R \times [0,M] : x+ie^{ia}\bar t \notin B_R\}; \\
E_3:= & \{(x,a) \in (\Omega\setminus B_R)\times [0,M]: x +ie^{ia}\bar t \in B_R\}.
\end{split}
\end{equation*}

For every $(x,a) \in E_1$ we define $\gamma_{\bar t,x,a}:[0,\bar t] \to B_R\times [0,M]$ by
\begin{equation*}
\gamma_{\bar t, x,a}(t) = 
\begin{cases}
(x+ie^{ia}t,a) & \mbox{if }t \in [0,\bar t), \\
T(x+ie^{ia}\bar t, a) & \mbox{if }t=\bar t,
\end{cases}
\end{equation*}
where the transport map $T$ is defined in Corollary \ref{C_T}.
For every $(x,a) \in E_2$ we set
\begin{equation*}
t^+(x,a):=\sup\{t \in [0,\bar t]: x+ie^{ia}t \in B_R\}
\end{equation*}
and we define 
$\gamma_{\bar t,x,a}:[0,t^+(x,a)) \to B_R\times [0,M]$ by
\begin{equation*}
\gamma_{\bar t,x,a}(t)=(x+ie^{ia}t,a).
\end{equation*}
For every $(x,a)\in E_3$ we set
\begin{equation*}
t^-(x,a):=\inf\{t \in [0,\bar t]: x+ie^{ia}t \in B_R\}
\end{equation*}
and we define $\gamma_{\bar t,x,a}:(t^-(x,a),\bar t] \to B_R\times [0,M]$ by
\begin{equation*}
\gamma_{\bar t,x,a}(t)=
\begin{cases}
(x+ie^{ia}t,a) & \mbox{if }t \in (t^-(x,a),\bar t) \\
T(x+ie^{ia}\bar t, a) & \mbox{if }t=\bar t.
\end{cases}
\end{equation*}

\subsubsection{Approximate characteristics}\label{Ss_approximate}

Fix $n \in N$ and set $\bar t_n=2^{-n}$. 
For every $(x,a)\in E_2$ we consider the curve 
\begin{equation*}
\gamma^{0,n}_{x,a}:\left(t^-_{\gamma^{0,n}_{x,a}}, t^+_{\gamma^{0,n}_{x,a}}\right) \to B_R \times [0,M]
\end{equation*}
with
\begin{equation*}
t^-_{\gamma^{0,n}_{x,a}}=0, \quad t^+_{\gamma^{0,n}_{x,a}}=t^+(x,a) \quad \mbox{and} \quad \gamma^{0,n}_{x,a}(t)=\gamma_{2^{-n},x,a}(t) \quad \forall t \in \left(t^-_{\gamma^{0,n}_{x,a}}, t^+_{\gamma^{0,n}_{x,a}}\right).
\end{equation*}

For every $(x,a) \in E_1$ we define 
\begin{equation*}
\gamma^{0,n}_{x,a}:\left(t^-_{\gamma^{0,n}_{x,a}}, t^+_{\gamma^{0,n}_{x,a}}\right) \to B_R \times [0,M]
\end{equation*}
with
\begin{equation*}
t^-_{\gamma^{0,n}_{x,a}}=0, \qquad t^+_{\gamma^{0,n}_{x,a}} \ge 2^{-n}
\end{equation*}
to be determined in the construction and 
\begin{equation*}
\gamma^{0,n}_{x,a}(t) = \gamma_{2^{-n},x,a}(t) \qquad \forall t \in \left(t^-_{\gamma^{0,n}_{x,a}}, 2^{-n}\right].
\end{equation*}

For every $k=1,\ldots, 2^n$ and for every $(x,a) \in E_3$ we introduce a curve
\begin{equation*}
\gamma^{k,n}_{x,a}:\left(t^-_{\gamma^{k,n}_{x,a}}, t^+_{\gamma^{k,n}_{x,a}}\right) \to B_R \times [0,M]
\end{equation*}
with 
\begin{equation*}
t^-_{\gamma^{k,n}_{x,a}} = (k-1)2^{-n} + t^-(x,a), \qquad t^+_{\gamma^{k,n}_{x,a}} \ge k2^{-n}
\end{equation*}
to be determined and 
\begin{equation*}
\gamma^{k,n}_{x,a}(t) = \gamma_{2^{-n},x,a}(t-(k-1)2^{-n}) \qquad \forall t \in \left(t^-_{\gamma^{k,n}_{x,a}}, k2^{-n}\right].
\end{equation*}
It remains to define the evolution of the curves $\gamma^{0,n}_{x,a}$ for $(x,a)\in E_1$ and $t\ge 2^{-n}$ 
and of the curves $\gamma^{k,n}_{x,a}$ for $(x,a) \in E_3$ and $t \ge k2^{-n}$.
Let us fix $k=1,\ldots, 2^n$ and $(x,a) \in E_3$.
We define the evolution of $\gamma^{k,n}_{x,a}$ by recursion:
assume that $\gamma^{k,n}_{x,a}$ is defined on $(t^-_{\gamma^{k,n}_{x,a}}, l2^{-n}]$ for some $l \ge k$.
If $l=2^n$ we set $t^+_{\gamma^{k,n}_{x,a}}=1$ otherwise, if $l<2^n$ we distinguish two cases.

If $\gamma^{k,n}_{x,a}(l2^{-n}) \in E_2$, then we set 
\begin{equation*}
t^+_{\gamma^{k,n}_{x,a}}= l2^{-n} + t^+\left(\gamma^{k,n}_{x,a}(l2^{-n})\right)
\end{equation*}
and
\begin{equation*}
\gamma^{k,n}_{x,a}(t)= \gamma_{2^{-n},\gamma^{k,n}_{x,a}(l2^{-n)}}(t-l2^{-n}) \qquad \forall t \in \left(l2^{-n},t^+_{\gamma^{k,n}_{x,a}}\right).
\end{equation*}
If instead $\gamma^{k,n}_{x,a}(l2^{-n}) \in E_1$, then we extend $\gamma^{k,n}_{x,a}$ on the whole interval
$(l2^{-n},(l+1)2^{-n}]$ by setting
\begin{equation*}
\gamma^{k,n}_{x,a}(t) = \gamma_{2^{-n},\gamma^{k,n}_{x,a}(l2^{-n})}(t-l2^{-n}) \qquad \forall t \in (l2^{-n},(l+1)2^{-n}].
\end{equation*}
The extension of the curves $\gamma^{0,n}_{x,a}$ for $(x,a)\in E_1$ is defined by the same procedure described above for the curves $\gamma^{k,n}_{x,a}$ for $(x,a) \in E_3$ with $k=1$.

\subsection{Approximate Lagrangian representation}
The approximate characteristics built in the previous section belong to the space
\begin{equation*}
\tilde \Gamma:=\left\{(\gamma,t^-_\gamma,t^+_\gamma) : 0\le t^-_\gamma \le t^+_\gamma\le 1, \gamma \in \BV\big((t^-_\gamma,t^+_\gamma);B_R\times[0,M]\big)\right\}.
\end{equation*}
For every $n\in \N$ sufficiently large we define $\omega_n \in \M(\tilde \Gamma)$ by
\begin{equation}\label{E_def_omega_n}
\omega_n = \int_{(B_R\times [0,M])\cap H_\phi}\delta_{\gamma^{0,n}_{a,x},t^-_{\gamma^{0,n}_{a,x}}, t^+_{\gamma^{0,n}_{a,x}}} dx da + \sum_{k=1}^{2^n}\int_{E_3\cap H_\phi} \delta_{\gamma^{k,n}_{a,x},t^-_{\gamma^{k,n}_{a,x}}, t^+_{\gamma^{k,n}_{a,x}}} dx da,
\end{equation}
where the curves $\gamma^{k,n}_{x,a}$ are defined in Section \ref{Ss_approximate}.

\begin{lemma}\label{L_hor_ver}
Let $\omega_n$ be defined in \eqref{E_def_omega_n}. Then the following estimates hold:
\begin{equation}\label{E_est_char1}
e_h(n):= \int_{\tilde \Gamma}\sup_{t \in (t^-_\gamma,t^+_\gamma)} \left| \gamma_x(t) - \gamma_x(t^-_\gamma) - 
\int_{t^-_\gamma}^t i e^{i\gamma_a(s)}ds \right| d\omega_n(\gamma)   = o(1) \qquad \mbox{as }n\to \infty
\end{equation}
\begin{equation}\label{E_est_vert}
e_v(n):=\int_{\tilde \Gamma}\TV_{(t^-_\gamma,t^+_\gamma)}\gamma_a d\omega_n(\gamma)  \le \nu(B_R) + o(1) \qquad \mbox{as }n\to \infty.
\end{equation}
\end{lemma}
\begin{proof}
 Since for $\omega_n$-a.e. $(\gamma,t^-_\gamma,t^+_\gamma) \in \tilde \Gamma$ it holds
 \begin{equation*}
 \dot\gamma_x(t) = ie^{i\gamma_a(t)} \qquad \forall t \in (\gamma,t^-_\gamma,t^+_\gamma)\setminus 2^{-n}\N,
 \end{equation*}
 then we have 
 \begin{equation}\label{E_single}
 \begin{split}
\sup_{t \in (t^-_\gamma,t^+_\gamma)} \left| \gamma_x(t) - \gamma_x(t^-_\gamma) - \int_{t^-_\gamma}^t i e^{i\gamma_a(s)}ds \right|  
 \le & ~\sum_{l^-(\gamma)}^{l^+(\gamma)} |\gamma_x(l2^{-n})-\gamma_x(l2^{-n}-)| \\
 = &~ \sum_{l^-(\gamma)}^{l^+(\gamma)} |T_x(\gamma(l2^{-n}-)-\gamma_x(l2^{-n}-)|,
\end{split}
 \end{equation}
 where $l^-(\gamma)= \inf ( 2^{-n}\Z \cap (t^-_\gamma,t^+_\gamma))$ and $l^+(\gamma)= \sup ( 2^{-n}\Z \cap (t^-_\gamma,t^+_\gamma))$.
 
 Integrating \eqref{E_single} with respect to $\omega_n$, it follows by  Corollary \ref{C_T} with $\bar t = 2^{-n}$ that 
\begin{equation}\label{E_est_char2}
\begin{split}
e_h(n)
 \le & ~ \sum_{l=1}^{2^n-1}\int_{X} |T_x(x,a)-x|   d (e_{l2^{-n}-})_\sharp \omega_n \\
 \le &~ \sum_{l=1}^{2^n-1}\left(\int_{X} |T_x(x,a)-x|  \rho^2_{\bar t}(x,a)dxda  +
2R\| ((e_{l2^{-n}-})_\sharp \omega_n- \rho^2_{\bar t}\L^3)^+  \|\right) \\
 \le &~ \frac{2^n}{L_{2^{-n}}} \left(2^{-n} + 2^{\frac{-3n}{2}}\right)\nu(B_R) + \e_{2^{-n}}^{\frac{1}{2}} \left( 2R + \e_{2^{-n}}^{\frac{1}{2}}M\right)+2R \sum_{l=1}^{2^n-1}\|  ((e_{l2^{-n}-})_\sharp \omega_n- \rho^2_{\bar t}\L^3)^+ \|,
\end{split}
\end{equation}
where $X=B_R\times [0,M]$ and $e_{t-}:\tilde \Gamma(t) \to X$ is defined by $e_{t-}(\gamma)= \lim_{t'\to t-}\gamma(t')$.
Since by construction $(e_{l2^{-n}-})_\sharp \omega_n \le \chi^2(\bar t)\L^3$ and $\rho^2_{\bar t}\le \chi^2(\bar t)$ with
$\| (\chi^2(\bar t)-\rho^2_{\bar t})\L^3\| \le 2^{-n}\e_{2^{-n}}$, then for every $l=1,\ldots, 2^n-1$
 it holds
 \begin{equation}\label{E_positive_part}
 \|  ((e_{l2^{-n}-})_\sharp \omega_n- \rho^2_{\bar t}\L^3)^+ \| \le  2^{-n}\e_{2^{-n}}.
 \end{equation}
 Plugging \eqref{E_positive_part} into \eqref{E_est_char2}, we immediately get \eqref{E_est_char1}.

We now prove \eqref{E_est_vert}. Since $\gamma_a$ is constant in each connected component of
$(t^-_\gamma, t^+_\gamma) \setminus 2^{-n}\N$ for $\omega_n$-a.e. $\gamma$, it follows by Corollary \ref{C_T} that 
\begin{equation*}
\begin{split}
\int_{\tilde \Gamma}\TV_{(t^-_\gamma,t^+_\gamma)}\gamma_a d\omega_n(\gamma) = &~ 
\sum_{l=1}^{2^n-1} \int_{\tilde \Gamma(l2^{-n})}|T_a(\gamma(l2^{-n}-)-\gamma_a(l2^{-n}-)|d\omega_n(\gamma) \\
= &~ \sum_{l=1}^{2^n-1} \int_X |T_a(x,a)-a| d (e_{l2^{-n}-})_\sharp \omega_n \\
\le &~ \sum_{l=1}^{2^n-1} \int_X |T_a(x,a)-a| \rho^2_{\bar t}(x,a)dxda + M \|\big((e_{l2^{-n}-})_\sharp \omega_n - \rho^2_{\bar t} \L^3\big)^+\| \\
\le &~ 2^n \left[\left(2^{-n} + 2^{\frac{-3n}{2}}\right)\nu(B_R) + \e_{2^{-n}}^{\frac{1}{2}}\bar t \left( 2R + \e_{2^{-n}}^{\frac{1}{2}}M\right)\right] + M\e_{2^{-n}},
\end{split}
\end{equation*}
which implies \eqref{E_est_vert}.
\end{proof}

We now show that $(e_t)_\sharp \omega_n$ approximates $\chi \L^3$ in the strong
topology of measures for every $t \in 2^{-n}\N \cap [0,1)$. This property and the weak continuity estimate
provided in Proposition \ref{P_L1_estimate} will guarantee Property (1) in Definition \ref{D_Lagr}.

\begin{lemma}\label{L_cumulative}
For every $l=0,\ldots, 2^n-1$ it holds
\begin{equation}\label{E_est_l}
\| (e_{l2^{-n}})_\sharp \omega_n - \chi \L^3\| \le  2^{-n+1}l\e_{2^{-n}}.
\end{equation}
Moreover for every $t \in [l2^{-n},(l+1)2^{-n})$ and every $\psi \in C^\infty_c(B_R\times [0,M])$ it holds
\begin{equation}\label{E_weak_cont}
\left| \int_X \psi d (e_t)_\sharp \omega_n - \int_X \psi d (e_{l2^{-n}})_\sharp \omega_n  \right|\le 
2^{-n}\left( 2M \H^1(\partial B_R) \|\psi\|_{L^\infty} + \|\nabla \psi\|_{L^\infty} \L^3(H_\phi) \right).
\end{equation}
\end{lemma}
\begin{proof}
The case $l=0$ follows by the definition of $\omega_n$. In order to get \eqref{E_est_l} we prove that for every 
$l=0,\ldots, 2^n-2$ it holds
\begin{equation*}
\| (e_{(l+1)2^{-n}})_\sharp \omega_n - \chi \L^3\| \le \| (e_{l2^{-n}})_\sharp \omega_n - \chi \L^3\|
+2^{-n+1} \e_{2^{-n}}.
\end{equation*}
Indeed
\begin{equation*}
\begin{split}
\| (e_{(l+1)2^{-n}})_\sharp \omega_n - \chi \L^3\| \le &~ 
\| (e_{(l+1)2^{-n}})_\sharp \omega_n - \rho^1_{\bar t} \L^3\| + \| \rho^1_{\bar t} \L^3 - \chi \L^3\| \\
= &~ \| T_\sharp (e_{(l+1)2^{-n}-})_\sharp \omega_n - T_\sharp (\rho^2_{\bar t} \L^3)\| + \| \rho^1_{\bar t} \L^3 - \chi \L^3\| \\
\le &~ \|(e_{(l+1)2^{-n}-})_\sharp \omega_n -  \rho^2_{\bar t} \L^3\| + 2^{-n}\e_{2^{-n}} \\
\le &~  \|(e_{(l+1)2^{-n}-})_\sharp \omega_n - \chi^2(\bar t)\L^3)\| + \| (\chi^2(\bar t)-\rho^2_{\bar t}\L^3)\|  + 2^{-n}\e_{2^{-n}} \\
\le &~  \| (e_{l2^{-n}})_\sharp \omega_n - \chi \L^3\| + 2 \cdot 2^{-n}\e_{2^{-n}}.
\end{split}
\end{equation*}
Inequality \eqref{E_weak_cont} follows by
\begin{equation*}
\begin{split}
\left| \int_X \psi d (e_t)_\sharp \omega_n - \int_X \psi d (e_{l2^{-n}})_\sharp \omega_n  \right|\le &~ 
\left| \int_X \psi d (e_t)_\sharp \omega_n\llcorner \{t^-_\gamma>l2^{-n}\} \right| + \left| \int_X \psi d (e_{l2^{-n}})_\sharp \omega_n \llcorner \{t^+_\gamma < t\}  \right| \\
&~ + \left| \int_X \psi d (e_t)_\sharp \omega_n\llcorner \{t^-_\gamma\le l2^{-n}\} - 
 \int_X \psi d (e_{l2^{-n}})_\sharp \omega_n \llcorner \{t^+_\gamma > t\}  \right| \\
 \le &~ 2\cdot 2^{-n}\H^1(\partial B_R) M \|\psi\|_{L^\infty} + \|\nabla \psi\|_{L^\infty} 2^{-n}\omega_n(\tilde \Gamma(t)) \\
 \le &~ 2^{-n+1}\H^1(\partial B_R) M \|\psi\|_{L^\infty} + \|\nabla \psi\|_{L^\infty} 2^{-n}\L^3(H_\phi). \qedhere
\end{split}
\end{equation*}
\end{proof}

\subsection{Compactness of $\omega_n$ and existence of a Lagrangian representation}
We consider on $\tilde \Gamma$ the topology $\tau$ that induces the following convergence:
$(\gamma_n,t^-_{\gamma_n},t^+_{\gamma_n})$ converges to $(\gamma,t^-_\gamma,t^+_\gamma)$ if $t^\pm_{\gamma_n} \to t^\pm_\gamma$ with respect to the Euclidean topology in $\R$ and there exist extensions $\tilde \gamma,\tilde \gamma_n$ of $\gamma, \gamma_n$ defined on $(0,1)$ such that the horizonal components
$\tilde \gamma_{n,x}$ converge to $\tilde \gamma_x$ uniformly and the vertical components $\tilde \gamma_{n,a}$ converge to $\tilde \gamma_a$ in $L^1(0,1)$. 
\begin{lemma}
The sequence of measures $\omega_n$ defined in \eqref{E_def_omega_n} is bounded and tight in $\M(\tilde \Gamma)$, 
namely for every $\e>0$ there exists $K_\e \subset \tilde \Gamma$ such that for every $n\in \N$ it holds
\begin{equation*}
\omega_n\left(\tilde \Gamma \setminus K_\e\right) <\e.
\end{equation*}
\end{lemma}
\begin{proof}
We prove first that the sequence $\omega_n$ is bounded: for every $n$ it holds
\begin{equation*}
|E_{3,n}\cap H_\phi| \le |E_{3,n}| \le M \H^1(\partial B_R) 2^{-n}.
\end{equation*}
In particular
\begin{equation*}
\limsup_{n\to \infty}|\omega_n|(\tilde \Gamma) = \limsup_{n\to \infty} \L^3(H_\phi)+2^n|E_{3,n}\cap H_\phi| \le  \L^3(H_\phi) + M \H^1(\partial B_R).
\end{equation*}


In order to prove the tightness of the sequence $\omega_n$ we consider for every $n\in \N$ and $C>0$ the set of curves $(\gamma,t^-_\gamma,t^+_\gamma)\in \tilde \Gamma_{n,C}\subset \tilde \Gamma$ satisfying the following properties:
\begin{enumerate}
\item $\TV_{(t^-_\gamma,t^+_\gamma)}\gamma_a \le C$;
\item $\sum_{k=l^-(\gamma)}^{l^+(\gamma)} |\gamma_x(2^{-n}k)-\gamma_x(2^{-n}k-)| \le C e_h(n)^{1/2}$, where $e_h(n)$ is defined in Lemma \ref{L_hor_ver}, $l^-(\gamma):= \inf 2^{-n}\Z \cap (t^-_\gamma,t^+_\gamma)$ and $l^+(\gamma):= \sup 2^{-n}\Z \cap (t^-_\gamma,t^+_\gamma)$;
\item $\Lip \gamma\llcorner ([(k-1)2^{-n},k2^{-n)})\le 1$ for every $k = l^-(\gamma),\ldots, l^+(\gamma)$.
\end{enumerate}
Since $e_h(n)$ tends to 0 as $n\to \infty$, for every $C>0$ the space 
\begin{equation*}
\tilde \Gamma(C):= \bigcup_{n=1}^\infty \tilde \Gamma_{n,C}
\end{equation*}
is compact with respect to the topology $\tau$ introduced above.
Moreover it follows by Lemma \ref{L_hor_ver} and Chebychev inequality that for every $\e>0$ there exists $C>0$ sufficiently large such that for every $n \in \N$
\begin{equation*}
\omega_n(\tilde \Gamma \setminus \tilde \Gamma(C))\le \e. \qedhere
\end{equation*}
 
%
%
%
\end{proof}

By Theorem \ref{T_Prok} it follows that the sequence $\omega_n$ is precompact with respect to the narrow convergence. We show in the next lemma that every limit point of $\omega_n$ is a Lagrangian representation of the hypograph of $\phi$ on $B_R$.

\begin{lemma}
Every limit point $\omega$ of the sequence $\omega_n$ is a Lagrangian representation of the hypograph of $\phi$ on $B_R$. 
\end{lemma}
\begin{proof}
We need to check that the three conditions in Definition \ref{D_Lagr} are satisfied and that $\omega \in \M_+(\Gamma)$, namely that $\omega$ is concentrated on $\Gamma$.
\newline
\emph{Condition (1)}. We prove that for every $t\in (0,1)$ the following two limits hold in the 
sense of distributions
\begin{equation}\label{E_two_lim}
\lim_{n\to \infty}(e_t)_\sharp \omega_n = \L^3\llcorner H_\phi, \qquad \mbox{and} \qquad \lim_{n\to \infty}(e_t)_\sharp \omega_n  = (e_t)_\sharp \omega.
\end{equation}
For every $t=2^{-k}\N \cap (0,1)$ for some $k \in \N$ the first limit holds true thanks to Lemma \ref{L_cumulative}, since 
$\chi \L^3 = \L^3\llcorner H_\phi$ by definition of $\chi$. 
The continuity in time stated in \eqref{E_weak_cont} implies that the limit holds true therefore for every $t \in (0,1)$ in the sense of distributions.
We observe that the second limit in \eqref{E_two_lim} is not trivial since $ e_t$ is not continuous on $\tilde \Gamma$ with respect to the topology introduced above. In order to establish it we need to check that for every $\psi \in C^\infty_c(B_R\times [0,M])$ it holds
\begin{equation*}
\lim_{n\to \infty} \int_{\tilde \Gamma(t)} \psi(\gamma(t))d\omega_n = \int_{\tilde \Gamma(t)} \psi(\gamma(t))d\omega.
\end{equation*}
Let $I\subset (0,1)$ be a non-empty open interval. 
Then consider the continuous and bounded function $T_{\psi,I}:\tilde \Gamma \to \R$ defined by
\begin{equation*}
T_{\psi,I}(\gamma,t^-_\gamma,t^+_\gamma) := \int_{I\cap (t^-_\gamma,t^+_\gamma)}\psi(\gamma(t))dt.
\end{equation*}
By definition of narrow convergence and Fubini theorem it follows that
\begin{equation*}
\lim_{n\to \infty}  \int_I \int_{\tilde \Gamma(t)}\psi (\gamma(t)) d\omega_n dt = \lim_{n\to \infty} \int_{\tilde \Gamma} T_{\psi, I} d\omega_n
= \int_{\tilde \Gamma} T_{\psi, I} d\omega =  \int_I \int_{\tilde \Gamma(t)}\psi (\gamma(t)) d\omega dt.
\end{equation*}
This proves that the second limit in \eqref{E_two_lim} holds for $\L^1$-a.e. $t \in (0,1)$. 
In order to prove that the limit is valid for every $t \in (0,1)$, we observe that $\omega$ is concentrated on curves with endpoints in $\partial B_R$ and for every $t\in (0,1)$ it holds
\begin{equation*}
\omega \left(\{(\gamma,t^-_\gamma,t^+_\gamma) \in \tilde \Gamma: t \in (t^-_\gamma,t^+_\gamma) \mbox{ and }\gamma_a(t-)\ne \gamma_a(t+)\}\right) = 0.
\end{equation*}
In particular $t \mapsto (e_t)_\sharp \omega$ is continuous in the sense of distributions on $B_R\times [0,M]$ and therefore the second limit in  \eqref{E_two_lim} holds for every $t\in (0,1)$.
\newline
\emph{Condition (2)}. 
The function $g:\tilde \Gamma\to \R$ defined by
\begin{equation*}
g(\gamma,t^-_\gamma,t^+_\gamma):=  \sup_{t\in (t^-_\gamma,t^+_\gamma)}\left|\gamma_x(t)-\gamma_x(t^-_\gamma) - \int_{t^-_\gamma}^t ie^{i\gamma_a(s)}ds\right|
\end{equation*}
is lower semicontinuous, therefore
\begin{equation*}
\int_{\tilde\Gamma}g(\gamma) d \omega \le \lim_{n\to \infty}\int_{\tilde\Gamma}g(\gamma) d \omega_n, 
\end{equation*}
which is equal to 0 by \eqref{E_est_char1}. 
\newline
\emph{Condition (3)} follows similarly from \eqref{E_est_vert}. In particular $\omega$ is concentrated on $\Gamma$ and this concludes the proof.
\end{proof}


\subsection{Representation of the defect measure and good curves selection}
In the following proposition we show that the kinetic measure $U_\phi$ can be decomposed along the characteristic trajectories detected by the Lagrangian representation $\omega_h$. 

\begin{proposition}
Let $\omega_h$ be a Lagrangian representation of the hypograph of $\phi$ on $B_R$ obtained as limit point of $\omega_n$ as in the previous section. Then 
\begin{equation*}
\L^1\times U_\phi = \int_{ \Gamma} \mu_\gamma d\omega_h(\gamma), \qquad \mbox{and} \qquad\L^1\times  |U_\phi| = \int_{ \Gamma} |\mu_\gamma| d\omega_h(\gamma).
\end{equation*}
\end{proposition}
\begin{proof}
Let $\bar \psi (t,x,a) =\varphi(t) \psi(x,a)\in C^\infty_c((0,1)\times B_R\times[0,M])$. Then
\begin{equation}\label{E_dec1}
\begin{split}
-\int_{(0,1)\times X} \varphi\partial_a\psi d U_\phi dt = &~ \int_{(0,1)\times H_\phi}ie^{ia}\cdot \nabla_x \psi \varphi dx da dt \\
= &~ \int_0^1 \int_{ \Gamma(t)}ie^{i\gamma_a(t)}\cdot \nabla_x \psi (\gamma(t))\varphi(t)d\omega_h(\gamma) dt \\
= &~  \int_0^1 \int_{ \Gamma(t)} \dot{\gamma_x}(t) \cdot \nabla_x \psi (\gamma(t))d\omega_h(\gamma) \varphi(t)dt \\
= &~ \int_{ \Gamma}\int_{t^-_\gamma}^{t^+_\gamma}   \dot{\gamma_x}(t)\cdot  \nabla_x \psi (\gamma(t)) \varphi(t) dt d\omega_h(\gamma).
\end{split}
\end{equation}
For every $\gamma \in \Gamma$, we consider the map $\psi_\gamma:= \psi \circ \gamma: (t^-_\gamma,t^+_\gamma)  \to B_R\times [0,M]$.
Since $\omega_h$-a.e. $\gamma \in  \Gamma$ has bounded variation on its domain, also $\psi_\gamma \in \BV((t^-_\gamma,t^+_\gamma);\R)$ and we have the following chain rule:
\begin{equation}\label{E_chain_rule}
\begin{split}
D_t \psi_\gamma =  &~ \nabla \psi (\gamma(t))\cdot \tilde D_t \gamma + \sum_{t_j\in J_\gamma}(\psi(\gamma(t_j+))-\psi(\gamma(t_j-)))\delta_{t_j} \\
= &~ \nabla_x \psi (\gamma(t))\cdot \tilde D_t \gamma_x + \partial_a \psi(\gamma(t))\tilde D_t\gamma_a + \sum_{t_j\in J_\gamma}(\psi(\gamma(t_j+))-\psi(\gamma(t_j-)))\delta_{t_j}.
\end{split} 
\end{equation}
Since for $\omega$-a.e. $\gamma$ it holds $D_t\gamma_x = \dot \gamma_x(t)\L^1$, plugging \eqref{E_chain_rule} into \eqref{E_dec1}, we obtain
\begin{equation*}\label{E_dec2}
\begin{split}
-\int_{(0,1)\times X}\varphi \partial_a\psi d U_\phi dt = &~  \int_{ \Gamma}\left(\int_{(t^-_\gamma, t^+_\gamma)} \varphi\left(D_t\psi_\gamma - 
\partial_a \psi(\gamma(t)) \tilde D_t\gamma_a - \sum_{t_j \in J_\gamma}  (\psi_\gamma(t_j+)-\psi_\gamma(t_j-))\delta_{t_j}\right)\right)d\omega_h \\
= &~ \int_{ \Gamma}\left(\int_{(t^-_\gamma, t^+_\gamma)} \varphi\left(D_t\psi_\gamma - 
\partial_a \psi(\gamma(t))\tilde D_t\gamma_a\right) - \sum_{t_j \in J_\gamma} \varphi(t_j) (\psi_\gamma(t_j+)-\psi_\gamma(t_j-))\right)d\omega_h.
\end{split}
\end{equation*}
We observe that by construction if $t^-_\gamma>0$, then $\gamma(t^-_\gamma) \in \partial B_R \times [0,M]$ and therefore $\psi(\gamma(t^-_\gamma))=0$. Similarly, if $t^+_\gamma<1$ then $\psi(\gamma(t^+_\gamma))=0$.
Therefore 
\begin{equation*}
\begin{split}
\int_{\Gamma}\int_{(t^-_\gamma,t^+_\gamma)}\varphi(t)D_t\psi_\gamma d\omega_h(\gamma) = &~ -\int_{\Gamma}\int_{(t^-_\gamma,t^+_\gamma)} \varphi'(t)\psi(\gamma(t))dt d\omega_h(\gamma)\\
= &~ \int_0^1\int_{H_\phi} \varphi'(t)\psi(x,a)dxdadt  \\
=&~ 0.
\end{split}
\end{equation*}
Since for $\omega_h$-a.e. $\gamma$ and every $t_j \in J_\gamma$ it holds $\gamma_x(t_j+)=\gamma_x(t_j-)$, it follows from the definition of $\mu_\gamma$ that
\begin{equation*}
\begin{split}
-\int_{(0,1)\times X}\varphi \partial_a\psi d U_\phi dt= &~ \int_{ \Gamma}\left(\int_{(t^-_\gamma, t^+_\gamma)} \left( - \varphi(t)
\partial_a \psi(\gamma(t))\tilde D_t\gamma_a\right) - \sum_{t_j \in J_\gamma} \varphi(t_j) (\psi(\gamma(t_j+))-\psi(\gamma(t_j-)))\right)d\omega_h \\
=& ~ - \int_{ \Gamma}\int_{(0,1)\times X} \varphi\partial_a \psi d \mu_\gamma d \omega_h(\gamma).
\end{split}
\end{equation*}
This proves the first equality in the statement when tested with functions of the form $\varphi\partial_a\psi$ for two test functions $\varphi,\psi$. Since both $U_\phi$ and $\int\mu_\gamma d\omega_h$ are supported on $[0,1]\times B_R \times [0,M]$ the equality holds true for every test function.

The inequality
\begin{equation*}
\L^1\times |U_\phi| \le \int_{ \Gamma}|\mu_\gamma| d\omega_h
\end{equation*}
follows immediately from the already proved first equality in the statement. In order to prove the opposite inequality it is enough to 
prove the global inequality
\begin{equation*}
\left(\L^1\times |U_\phi|\right)((0,1)\times B_R\times [0,M])\ge \int_{ \Gamma}|\mu_\gamma|((0,1) \times B_R\times [0,M])d\omega_h.
\end{equation*}
We observe that $|\mu_\gamma|((0,1) \times B_R\times [0,M])= \TV_{(t^-_\gamma, t^+_\gamma)} \gamma_a$ and that the map
\begin{equation*}
(\gamma,t^-_\gamma,t^+_\gamma) \mapsto \TV_{(t^-_\gamma, t^+_\gamma)} \gamma_a
\end{equation*}
is lower semicontinuous on $\tilde \Gamma$. Therefore it follows from \eqref{E_est_vert} that
\begin{equation*}
\begin{split}
\int_{ \Gamma}|\mu_\gamma|(B_R\times [0,M])d\omega_h = &~ \int_{ \Gamma}\TV_{(t^-_\gamma, t^+_\gamma)} \gamma_a d\omega_h \\
\le &~ \liminf_{n\to \infty} \int_{\tilde \Gamma}\TV_{(t^-_\gamma, t^+_\gamma)} \gamma_a d\omega_n \\
\le &~  \left(\L^1\times |U_\phi|\right)((0,1) \times B_R\times [0,M]). \qedhere
\end{split}
\end{equation*}
\end{proof}


With the result above the proof of the part of Theorem \ref{T_Lagrangian} concerning the hypograph of $\phi$ is complete;
the statement for the epigraph of $\phi$ can be proven in the same way.

The following lemma is an application of Tonelli theorem and it is already proven in \cite{M_Burgers} to which we refer for the details.

\begin{lemma}
For $\omega_h$-a.e. $\gamma \in  \Gamma$ it holds that for $\L^1$-a.e. $t\in (t^-_\gamma,t^+_\gamma)$
\begin{enumerate}
\item $\gamma_x(t)$ is a Lebesgue point of $\phi$;
\item $\gamma_a(t)< \phi(\gamma_x(t))$.
\end{enumerate}
We denote by $ \Gamma_h$ the set of curves $\gamma\in  \Gamma$ such that the two properties above hold.
Similarly for $\omega_e$-a.e. $\gamma \in  \Gamma$ it holds that for $\L^1$-a.e. $t\in (t^-_\gamma,t^+_\gamma)$
\begin{enumerate}
\item $\gamma_x(t)$ is a Lebesgue point of $\phi$;
\item $\gamma_a(t)> \phi(\gamma_x(t))$
\end{enumerate}
and we denote the set of these curves by $ \Gamma_e$.
\end{lemma}

%
%
%
%

\section{Rectifiability of the measure $\nu$}
In this section we prove that the measure $\nu:=(p_x)_\sharp |U_\phi|$ is concentrated on a 1-rectifiable set. 
The rectifiability of $\nu$ is equivalent to the rectifiability of both the measures $(p_x)_\sharp U_\phi^-$ and 
$(p_x)_\sharp U_\phi^+$. Being the two cases analogous we provide the proof of the rectifiability of $(p_x)_\sharp U_\phi^-$ only.

\subsection{Pairing between $\omega_h$ and $\omega_e$ and its decomposition}
In the following lemma we introduce a pairing between the two representations $\omega_h \otimes \mu^-_\gamma$ and 
$\omega_e \otimes \mu^+_\gamma$ of the negative part of the defect measure $\L^1 \times U^-_\phi$. We will denote by $X$ the set $B_R\times [0,M]$.

\begin{lemma}\label{L_pairing}
Denote by $p_1,p_2: ( \Gamma \times [0,1] \times  X)^2\to  \Gamma \times [0,1]\times X$ the standard projections. 
Then there exists a plan $\pi^- \in \M(( \Gamma \times [0,1]\times X)^2)$ with marginals
\begin{equation}\label{E_marginals2}
\begin{split}
(p_1)_{\sharp}\pi^- =   \omega_h\otimes \mu^-_\gamma, \\
(p_2)_\sharp \pi^-=   \omega_e\otimes \mu^+_\gamma,
\end{split}
\end{equation}
concentrated on the set
\begin{equation*}
\begin{split}
\mathcal G := \big\{((\gamma,t^-_\gamma,t^+_\gamma, t,x,a),(\gamma',{t^-_\gamma}',{t^+_\gamma}', t',x',a')) \in (\Gamma \times X)^2 : \, &t\in (t^-_\gamma,t^+_\gamma), t' \in ({t^-_\gamma}',{t^+_\gamma}'), t=t' \\
\gamma_x(t)=x=x'=\gamma'_x(t'), a=a',
&a \in [\gamma_a(t+),\gamma_a(t-)]\cap [\gamma'_a(t'-),\gamma'_a(t'+)]\big\}.
\end{split}
\end{equation*}
\end{lemma}
\begin{proof}
First we observe that by definition, $\omega_h\otimes \mu^-_\gamma$ is concentrated on the set
\begin{equation*}
\G^-_h:= \{(\gamma,t^-_\gamma,t^+_\gamma,t,x,a) \in  \Gamma \times [0,1] \times X: t \in (t^-_\gamma,t^+_\gamma), \gamma_x(t)=x, a \in [\gamma_a(t+),\gamma_a(t-)]\}
\end{equation*}
and $\omega_e\otimes \mu^+_\gamma$ is concentrated on the set
\begin{equation*}
\G^+_e:= \{(\gamma,t^-_\gamma,t^+_\gamma,t,x,a) \in  \Gamma \times [0,1] \times X: t \in (t^-_\gamma,t^+_\gamma), \gamma_x(t)=x, a \in [\gamma_a(t-),\gamma_a(t+)]\}.
\end{equation*}

Denoting by $p_{2,3}: \Gamma\times [0,1] \times X \to [0,1]\times X$ the standard projection it follows from \eqref{E_mu-} that
\begin{equation*}
(p_{2,3})_\sharp (\omega_h \otimes \mu^-_\gamma) = \L^1 \times U_\phi^- = (p_{2,3})_\sharp (\omega_e \otimes \mu^+_\gamma).
\end{equation*}
By the disintegration theorem (see for example \cite{AFP_book}) there exist two measurable families of probability measures 
$(\mu^{-,h}_{t,x,a})_{(t,x,a) \in X}, (\mu^{+,e}_{t,x,a})_{(t,x,a) \in X} \in \P( \Gamma\times [0,1] \times X)$ such that
\begin{equation}\label{E_disintegration2}
\omega_h \otimes \mu^-_\gamma = \int_{[0,1]\times X} \mu^{-,h}_{t,x,a} d\L^1\times U^-_\phi \qquad \mbox{and} \qquad \omega_e \otimes \mu^+_\gamma =  \int_{[0,1]\times X} \mu^{+,e}_{t,x,a} d \L^1\times U^-_\phi
\end{equation}
and for $\L^1\times U^-_\phi$-a.e. $(t,x,a)$ the measures $\mu^{-,h}_{t,x,a}$ and $\mu^{+,e}_{t,x,a}$ are concentrated on the set 
\begin{equation*}
p_{2,3}^{-1}(\{t,x,a\}) = \{(\gamma,t^-_\gamma,t^+_\gamma,t',x',a')\in  \Gamma \times [0,1] \times X: t'=t,x'=x,a'=a\}.
\end{equation*}
Moreover, since $\omega_h\otimes \mu^-_\gamma$ is concentrated on the set $\G^-_h$ and $\omega_e\otimes \mu^+_\gamma$ is concentrated on the set $\G^+_e$,
we have that for $\L^1\times U^-_\phi$-a.e. $(t,x,a)$ the measure $\mu^{-,h}_{t,x,a}$ is concentrated on $p_{2,3}^{-1}(\{t,x,a\}) \cap \G^-_h$ and $\mu^{+,e}_{t,x,a}$ is concentrated on 
$p_{2,3}^{-1}(\{t,x,a\}) \cap \G^+_e$.
We eventually set
\begin{equation*}
\pi^-:= \int_{[0,1]\times X} \left(\mu^{-,h}_{t,x,a} \otimes \mu^{+,e}_{t,x,a}\right) d(\L^1\times U^-_\phi).
\end{equation*}
From \eqref{E_disintegration2} it directly follows \eqref{E_marginals2} and by the above discussion for $\L^1\times U^-_\phi$-a.e. $(t,x,a) \in [0,1]\times X$ the measure $\mu^{-,h}_{t,x,v} \otimes \mu^{+,e}_{t,x,v}$
is concentrated on $(p_{2,3}^{-1}(\{t,x,a\}) \cap \G^-_h) \times (p_{2,3}^{-1}(\{t,x,a\}) \cap \G^+_e)$, therefore $\pi^-$ is concentrated on
\begin{equation*}
\bigcup_{(t,x,a) \in [0,1]\times X} (p_{2,3}^{-1}(\{t,x,a\}) \cap \G^-_h) \times (p_{2,3}^{-1}(\{t,x,a\}) \cap \G^+_e) = \G
\end{equation*}
and this concludes the proof.
\end{proof}

We now split the set $\G$ introduced in Lemma \ref{L_pairing} in finitely many components.
We first set 
\begin{equation*}
\begin{split}
\G^-_{h,\mathrm{jump}}:= &\left\{(\gamma,t^-_\gamma,t^+_\gamma,t,x,a) \in \G^-_h : \gamma_a(t+)< \gamma_a(t-)\right\},\\
\G^+_{e,\mathrm{jump}}:= &\left\{(\gamma,t^-_\gamma,t^+_\gamma,t,x,a) \in \G^+_e : \gamma_a(t-)< \gamma_a(t+)\right\}.
\end{split}
\end{equation*}
We moreover consider the following covering with overlaps of $[0,M]$. 
Let $L= \lfloor \frac{2M}{\pi}\rfloor$ and for every $l=0,\ldots, L$ set
\begin{equation*}
I_l = \left( l\frac{\pi}{2}-\frac{\pi}{8}, (l+1)\frac{\pi}{2} + \frac{\pi}{8}\right)
\end{equation*}
and 
\begin{equation*}
\begin{split}
\G^-_{h,l}:= &\left\{(\gamma,t^-_\gamma,t^+_\gamma,t,x,a) \in \G^-_h : \gamma_a(t+), \gamma_a(t-) \in I_l \right\}, \\
\G^+_{e,l}:= &\left\{(\gamma,t^-_\gamma,t^+_\gamma,t,x,a) \in \G^+_e : \gamma_a(t-), \gamma_a(t+)\in I_l\right\}.
\end{split}
\end{equation*}
We then define
\begin{equation*}
\pi^-_l := \pi^- \llcorner \left(  \G^-_{h,l} \times \G^+_{e,l} \right), \qquad \pi^-_{\mathrm{jump}}= \pi^-\llcorner  \left( \big(\G^-_{h,\mathrm{jump}} \times \G^+_e\big) \cup  \big(\G^-_h \times \G^+_{e,\mathrm{jump}}\big) \right).
\end{equation*}

We prove separately that $ \nu^-_{\mathrm{jump}}:=(p^1_x)_\sharp \pi^-_{\mathrm{jump}}$ is 1-rectifiable and that  $\nu^-_l:=(p^1_x)_\sharp \pi^-_l$ is rectifiable for every $l=0,\ldots,L$.

\subsection{Rectifiability of $\nu^-_l$}\label{Ss_nu_l}
The proof of the rectifiability of $\nu^-_l$ follows the strategy used in \cite{M_Burgers}. 
In particular the first step is to identify a countable family of Lipschitz curves where we will prove that $\nu^-_l$ is concentrated.
\subsubsection{Shock curves}
 For shortness we denote by
\begin{equation*}
e_l:= ie^{i\left(l\frac{\pi}{2}+\frac{\pi}{4}\right)} \qquad \mbox{and} \qquad e_l^\perp:= i e_l.
\end{equation*}

The following proposition establishes the intuitive fact that a curve of the epigraph cannot cross from below a curve of the hypograph.
Since the same proposition and the following corollary were proven in \cite{M_Burgers} in the case of Burgers equation, we only sketch the arguments here.
\begin{proposition}\label{P_no_crossing}
Let $(\bar \gamma,t^-_{\bar \gamma},t^+_{\bar \gamma}) \in  \Gamma_{h}$ and let $(\tilde t^-_{\bar \gamma},\tilde t^+_{\bar \gamma}) \subset (t^-_{\bar \gamma},t^+_{\bar \gamma})$ be such that 
\begin{equation*}
\gamma_a((\tilde t^-_{\bar \gamma},\tilde t^+_{\bar \gamma}))\subset I_l.
\end{equation*}
We denote by $G_{\mathrm{cr}}(\bar \gamma, \tilde t^-_{\bar \gamma},\tilde t^+_{\bar \gamma})$ the set of curves $(\gamma,t^-_\gamma,t^+_\gamma) \in  \Gamma_e$ for which $\exists \bar t_1,\bar t_2 \in  (\tilde t^-_{\bar \gamma},\tilde t^+_{\bar \gamma})$ and $t_1,t_2 \in (t^-_\gamma,t^+_\gamma)$ such that the following conditions are satisfied:
\begin{enumerate}
\item $t_1<t_2$ and $\bar t_1 < \bar t_2$;
\item $\gamma_a((t_1,t_2)) \subset I_l$;
\item $\gamma_x(t_1) \cdot e_l^\perp = \bar \gamma_x(\bar t_1)\cdot e_l^\perp$ and $ \gamma_x(t_1)\cdot e_l < \bar \gamma_x(\bar t_1) \cdot e_l$;
\item $\gamma_x(t_2) \cdot e_l^\perp = \bar \gamma_x(\bar t_2)\cdot e_l^\perp$ and $ \gamma_x(t_2)\cdot e_l > \bar \gamma_x(\bar t_2) \cdot e_l$.
\end{enumerate}
Then 
\begin{equation*}
\omega_e (G_{\mathrm{cr}}(\bar \gamma, \tilde t^-_{\bar \gamma},\tilde t^+_{\bar \gamma}))=0.
\end{equation*}
\end{proposition}
\begin{proof}
Let
\begin{equation*}
s^-:= \bar \gamma_x(\tilde t^-_{\bar \gamma})\cdot e_l^\perp \qquad \mbox{and} \qquad s^+:=\bar\gamma_x(\tilde t^+_{\bar \gamma})\cdot e_l^\perp. 
\end{equation*}
Since $\bar \gamma_a((\tilde t^-_{\bar \gamma},\tilde t^+_{\bar \gamma}))\subset I_l$ and $\dot{\overline{\gamma}}_x(t) = e^{i\bar \gamma_a(t)}$ for $\L^1$-a.e. $t \in (\tilde t^-_{\bar \gamma},\tilde t^+_{\bar \gamma})$, then the map
\begin{equation*}
\begin{split}
h_{\bar \gamma}: (\tilde t^-_{\bar \gamma},\tilde t^+_{\bar \gamma}) &\to (s^-,s^+) \\
t & \mapsto \bar \gamma_x(t)\cdot e_l^\perp
\end{split}
\end{equation*}
is bi-Lipschitz. For every $s \in (s^-,s^+)$ we set $g_{\bar \gamma}(s)= \gamma_x(h_{\bar \gamma}^{-1}(t))\cdot e_l$.
Let $\delta>0$ and $\psi_\delta: \R\to \R$ be the Lipschitz approximation of the Heaviside function defined by $\psi_\delta(v) = 0 \vee (v/\delta \wedge 1)$. Let us consider a measurable selection of $t_1,t_2$ in $ \Gamma_\mathrm{cr}(\bar \gamma, \tilde t^-_{\bar \gamma},\tilde t^+_{\bar \gamma})$ and let us denote by
\begin{equation*}
 \Gamma_\mathrm{cr}(\bar \gamma, \tilde t^-_{\bar \gamma},\tilde t^+_{\bar \gamma},\delta):=\left\{ (\gamma,t^-_\gamma,t^+_\gamma)\in  \Gamma_\mathrm{cr}(\bar \gamma, \tilde t^-_{\bar \gamma},\tilde t^+_{\bar \gamma}) : \gamma_x(t_{1,\gamma})\cdot e_l - g_{\bar \gamma}(h_{\bar \gamma}(\gamma_x(t)\cdot e_l^\perp))>\delta\right\}.
\end{equation*}
For every $t \in (0,1)$ and $\gamma \in  \Gamma_\mathrm{cr}(\bar \gamma, \tilde t^-_{\bar \gamma},\tilde t^+_{\bar \gamma},\delta)$ set
\begin{equation*}
f(\gamma,t):=
\begin{cases}
0 & \mbox{if }t<t_{1,\gamma}; \\
\psi_\delta(\gamma_x(t)\cdot e_l - g_{\bar \gamma}(\gamma_x(t)\cdot e_l^\perp)) & \mbox{if }t \in (t_{1,\gamma},t_{2,\gamma}); \\
1 & \mbox{if }t>t_{2,\gamma}.
\end{cases}
\end{equation*}
Finally we consider the functional
\begin{equation*}
\Psi_\delta(t):= \int_{ \Gamma_\mathrm{cr}(\bar \gamma, \tilde t^-_{\bar \gamma},\tilde t^+_{\bar \gamma},\delta)}f(\gamma,t)d\omega_e(\gamma).
\end{equation*}
A straightforward computation shows that 
\begin{equation}\label{E_psi_delta}
\Psi_\delta'(t) \le \frac{C}{\delta}\int_{G(\delta,t)}\left(\bar \gamma_a(h_{\bar \gamma}(\gamma_x(t)\cdot e_l^\perp))-\gamma_a(t)\right)^+ d\omega_e(\gamma),
\end{equation}
where
\begin{equation*}
\begin{split}
G(\delta,t) = \big\{  (\gamma, t^-_\gamma,t^+_\gamma) \in   \Gamma_\mathrm{cr}(\bar \gamma, \tilde t^-_{\bar \gamma},\tilde t^+_{\bar \gamma},\delta) :\, & t\in (t_{1,\gamma},t_{2,\gamma}) \mbox{ and }  \\ &\gamma_x(t)\cdot e_l \in (\bar \gamma(h_{\bar \gamma}(\gamma_x(t)\cdot e_l^\perp)), \bar \gamma(h_{\bar \gamma}(\gamma_x(t)\cdot e_l^\perp))+\delta)\big\}.
\end{split}
\end{equation*}
Let us denote by
\begin{equation*}
S_\delta := \{ x \in B_R : x\cdot e_l \in (g_{\bar \gamma}(x\cdot e_l^\perp),g_{\bar \gamma}(x\cdot e_l^\perp)+\delta) \}.
\end{equation*}
Since $ ( e_t)_\sharp \omega_e \llcorner G(\delta,t) \le E_\phi\cap (S_{\delta}\times [0,M])$ and for $\L^1$-a.e. $t \in  (\tilde t^-_{\bar \gamma},\tilde t^+_{\bar \gamma})$ the point $\bar \gamma_x(t)$ is a Lebesgue point of $\phi$ with value larger than $\bar \gamma_a(t)$ 
we obtain from \eqref{E_psi_delta} that $\Psi_\delta'(t) \le o(1)$ as $\delta \to 0$.
By definition of the functional $\Psi_\delta$ it holds
\begin{equation*}
\omega_e (  \Gamma_\mathrm{cr}(\bar \gamma, \tilde t^-_{\bar \gamma},\tilde t^+_{\bar \gamma}))\le \liminf_{\delta \to 0} \Psi_\delta(1) =0. \qedhere
\end{equation*} 
\end{proof}

\begin{corollary}\label{C_no_crossing}
Let $\bar x \in B_R$ and denote by $\Gamma^+_l(\bar x)$ the set of curves $(\gamma,t^-_\gamma,t^+_\gamma) \in \Gamma_h$ 
for which there exists $t_1 \in (t^-_\gamma, t^+_\gamma)$ such that
\begin{equation*}
 \gamma_x(t_1)\cdot e_l^\perp = \bar x \cdot e_l^\perp \qquad \mbox{and} \qquad  \gamma_x(t_1)\cdot e_l> \bar x \cdot e_l.
\end{equation*}
Similarly let $\Gamma^-_l(\bar x)$ the set of curves $(\gamma,t^-_\gamma,t^+_\gamma) \in \Gamma_e$ 
for which there exists $t_1' \in (t^-_\gamma, t^+_\gamma)$ such that
\begin{equation*}
 \gamma_x(t_1)\cdot e_l^\perp = \bar x \cdot e_l^\perp \qquad \mbox{and} \qquad  \gamma_x(t_1)\cdot e_l< \bar x \cdot e_l.
\end{equation*}
Then there exists a Lipschitz function $f_{\bar x,l}:[\bar x\cdot e_l^\perp,+\infty) \to \R$ such that 
\begin{equation}\label{E_cross}
\begin{split}
&\omega_h(\{(\gamma,t^-_\gamma,t^+_\gamma) \in \Gamma^+_l(\bar x) : \exists t_2 \in (t_1,t^+_\gamma) \mbox{ s.t. }\gamma_a((t_1,t_2))\subset I_l \mbox{ and }\gamma_x(t_2)\cdot e_l < f_{\bar x,l}(\gamma_x(t_2)\cdot e_l^\perp)  \})=0, \\
&\omega_e(\{(\gamma,t^-_\gamma,t^+_\gamma) \in \Gamma^-_l(\bar x) : \exists t_2' \in (t_1',t^+_\gamma) \mbox{ s.t. }\gamma_a((t_1',t_2'))\subset I_l \mbox{ and }\gamma_x(t_2')\cdot e_l > f_{\bar x,l}(\gamma_x(t_2')\cdot e_l^\perp)  \})=0,
\end{split}
\end{equation}
where $t_1,t_1'$ are as above.
\end{corollary}

\begin{proof}
Let $I\subset [\bar x,+\infty)$ be the set of values $y$ for which there exists $\gamma\in \Gamma^+_{l}(\bar x)$ and $t\in (t_1,t_2)$ such that $\gamma_x(t)\cdot e_l^\perp = y$, where $t_1$ and $t_2$ are as in the statement. Let $\tilde f_{\bar x,l}$ be defined on $I$ by
\begin{equation*}
\tilde f_{\bar x,l}(y):= \inf \left\{ \gamma_x(t)\cdot e_l : \gamma\in \Gamma^+_{l}(\bar x), t\in (t_1,t_2),  \gamma_x(t)\cdot e_l^\perp = y \right\}.
\end{equation*}
The function $f_{\bar x,l}$ is defined as the biggest $C$-Lipschitz function such that $f_{\bar x,l}\le\tilde f_{\bar x,l}$ on $I$ and $f_{\bar x,l}=\bar x \cdot e_l$, where $C>\tan(3\pi/8)$. 
The first equality in \eqref{E_cross} follows from the fact that $f_{\bar x,l}\le\tilde f_{\bar x,l}$ on $I$.
The second equality in \eqref{E_cross} follows from Proposition \ref{P_no_crossing} since the infimum in the definition of $\tilde f_{\bar x,l}$ can be realized taking only countably many curves in $\Gamma^+_{l}(\bar x)$ and for every $a \in I_l$ it holds 
\begin{equation*}
ie^{ia}\cdot e_l^\perp \ge \cos \left(\frac{3\pi}{8}\right). \qedhere
\end{equation*}
\end{proof}

The following elementary lemma is about functions of bounded variation of one variable: we refer to \cite{AFP_book} for the theory of BV functions.
\begin{lemma}\label{L_BV}
Let $v: (a,b)\to \R$ be a $\BV$ function and denote by $D^-v$ the negative part of the measure $Dv$. Then for $\tilde D^-v$-a.e. $\bar x \in (a,b)$ there exists $\delta >0$ such that
\begin{equation*}
\bar v(x) > \bar v(\bar x) \quad \forall x \in (\bar x - \delta, \bar x) \qquad \mbox{and} \qquad \bar v(x) < \bar v(\bar x) \quad \forall x \in (\bar x, \bar x +  \delta).
\end{equation*}
\end{lemma}

We are now in position to prove the rectifiability of $\nu^-_l$.
\begin{proposition}\label{P_l}
The measure $ \nu_l^-$ is concentrated on the set
\begin{equation*}
\bigcup_{\bar x \in \Q^2 \cap B_R} C_{f_{\bar x,l}}, \qquad \mbox{where} \qquad C_{f_{\bar x,l}}:=B_R \cap \bigcup_{s> \bar x \cdot e_l^\perp} \left\{se_l^\perp + f_{\bar x, l}(s) e_l \right\} 
\end{equation*}
\end{proposition}
\begin{proof}
Step 1. For every $\bar x \in B_R\cap \Q^2$ and every $(\gamma,t^-_\gamma,t^+_\gamma)\in \Gamma_h$ we consider the open set $I^+_{\bar x,l,\gamma}\subset (t^-_\gamma,t^+_\gamma)$ defined by the following property: we say that $t\in I^+_{\bar x,l,\gamma}$ if there exists $t'\in (t^-_\gamma,t)$ such that 
\begin{equation*}
\gamma_a((t',t))\subset I_l, \qquad \gamma_x(t')\cdot e_l^\perp=\bar x \cdot e_l^\perp, \qquad \gamma_x(t')\cdot e_l > \bar x \cdot e_l.
\end{equation*}
We moreover set
\begin{equation*}
\G^>_{\bar x,l}:= \left\{ (\gamma,t^-_\gamma,t^+_\gamma,t,x,a) \in \Gamma_h\times (0,1)\times B_R \times [0,M] : t \in I^+_{\bar x, l ,\gamma} \right\}.
\end{equation*}
Similarly for every $(\gamma,t^-_\gamma,t^+_\gamma) \in \Gamma_e$ we let $I^-_{\bar x,l,\gamma}\subset (t^-_\gamma,t^+_\gamma)$ be the set of $t$ for which $\exists t'\in (t^-_\gamma,t)$ such that 
\begin{equation*}
\gamma_a((t',t))\subset I_l, \qquad \gamma_x(t')\cdot e_l^\perp=\bar x \cdot e_l^\perp, \qquad \gamma_x(t')\cdot e_l < \bar x \cdot e_l
\end{equation*}
and we set
\begin{equation*}
\G^<_{\bar x,l}:= \left\{ (\gamma,t^-_\gamma,t^+_\gamma,t,x,a) \in \Gamma_e\times (0,1)\times B_R \times [0,M] : t \in I^-_{\bar x, l ,\gamma} \right\}.
\end{equation*}
%
%
We consider
\begin{equation*}
\pi^-_{\bar x,l}:= \pi^- \llcorner \left( \mathcal G^>_{\bar x,l} \times \mathcal G^<_{\bar x,l}\right)
\end{equation*}
and we prove that $(p^1_x)_\sharp \pi^-_{\bar x,l}$ is concentrated on $C_{f_{\bar x,l}}$, 
where 
\begin{equation*}
\begin{split}
p^1_x: ( \Gamma\times (0,1) \times B_R \times [0,M])^2 & \to B_R \\
(\gamma,t^-_\gamma,t^+_\gamma,t,x,a,\gamma',{t^-_\gamma}',{t^+_\gamma}', t',x',a') & \mapsto x.
\end{split}
\end{equation*}

Trivially it holds 
\begin{equation}\label{E_est_Omega+}
(p^1_x)_\sharp \pi^-_{\bar x,l} \le  
(p^1_x)_\sharp \big[\pi^- \llcorner \big( \mathcal G^>_{\bar x, l} \times ( \Gamma \times (0,1) \times B_R \times [0,M]\big) \big].
\end{equation}
From Corollary \ref{C_no_crossing} it follows that for $\omega_{h}$-a.e. $(\gamma,t^-_\gamma,t^+_\gamma) \in \Gamma_h$ it holds 
\begin{equation*}
\gamma_x(t)\cdot e_l^\perp > \bar x \cdot e_l^\perp \quad \mbox{and} \quad \gamma_x(t)\cdot e_l\ge f_{\bar x,l}(\gamma_x(t)\cdot e_l^\perp) \qquad \forall t \in I_{\bar x,l,\gamma},
\end{equation*}
therefore
\begin{equation}\label{E_null+}
(p^1_x)_\sharp \pi^-_{\bar x,l}\left( \left\{ x \in B_R: x\cdot e_l^\perp \le \bar x \cdot e_l^\perp\right\} \cup \left\{ x\in B_R: x\cdot e_l^\perp > \bar x \cdot e_l^\perp\mbox{ and } x\cdot e_l< f_{\bar x,l}(x\cdot e_l^\perp)\right\}\right)=0.
\end{equation}
%
%
In the same way we get 
\begin{equation}\label{E_null-}
(p^2_x)_\sharp \pi^-_{\bar x,l}\left(  \left\{ x \in B_R: x\cdot e_l^\perp \le \bar x \cdot e_l^\perp\right\} \cup \left\{ x\in B_R: x\cdot e_l^\perp > \bar x \cdot e_l^\perp\mbox{ and } x\cdot e_l> f_{\bar x,l}(x\cdot e_l^\perp)\right\}\right)=0,
\end{equation}
where
\begin{equation*}
\begin{split}
p^2_x: ( \Gamma\times (0,1) \times B_R \times [0,M])^2 & \to B_R \\
(\gamma,t^-_\gamma,t^+_\gamma,t,x,a,\gamma',{t^-_\gamma}',{t^+_\gamma}', t',x',a') & \mapsto x'.
\end{split}
\end{equation*}.

Finally, since $\pi^-$ is concentrated on $\G$, then
\begin{equation*}
(p^1_{x} \otimes p^2_{x})_\sharp \pi^- \in \M(([0,T]\times \R)^2)
\end{equation*}
is concentrated on the graph of the identity on $B_R$ and in particular $(p^1_{x})_\sharp \pi^-_{\bar x,l}= (p^2_{x})_\sharp \pi^-_{\bar x,l}$.
Therefore it follows from \eqref{E_null+} and \eqref{E_null-} that $(p^1_{x})_\sharp \pi^-_{ \bar x,l}$ is concentrated on
\begin{equation*}
\left\{x\in B_R:  x\cdot e_l > \bar x \cdot e_l \mbox{ and }x\cdot e_l^\perp = f_{\bar x,l}(x\cdot e_l) \right\}= C_{f_{\bar x,l}}.
\end{equation*}
Step 2. We prove that for $\pi^-_l$-a.e. $\mathcal Z =  (\gamma,t^-_\gamma,t^+_\gamma,t,x,a,\gamma',{t^-_\gamma}',{t^+_\gamma}',t',x',a') \in ( \Gamma\times (0,1) \times B_R \times [0,M])^2$ there exists $\delta >0$ such that for every $s \in (t-\delta,t)$ and $s'\in (t'-\delta,t')$ the following properties hold:
\begin{enumerate}
\item $\gamma_a(s)\in I_l$ and $\gamma_a(s) > a$;
\item $\gamma'_a(s')\in I_l$ and $\gamma'_a(s') < a'$.
\end{enumerate}
It is sufficient to prove the properties in (1), being the ones in (2) analogous. 
The statement is trivial for elements $\mathcal Z$ for which $\gamma_a(t-)>a$ and it follows immediately by Lemma \ref{L_BV} applied to $\gamma_a$ if $\gamma_a$ is continuous at $t$. Being $\pi^-$ concentrated on points $\mathcal Z$ for which $\gamma_a(t+)\le \gamma_a(t-)$ it is therefore sufficient to check that
\begin{equation*}
\pi^-_l \left( \left\{ \mathcal Z \in ( \Gamma\times (0,1) \times B_R \times [0,M])^2 : \gamma_a(t-)=a>\gamma_a(t+) \right\} \right)=0.
\end{equation*}
This follows immediately from the fact that for $\omega_h$-a.e. $\gamma$ the measure $\mu^-_\gamma$ has no atoms and the set
$(t,x,a) \in (0,1)\times B_R\times [0,M]$ for which $\gamma_x(t)=x$ and $\gamma_a(t-)=a>\gamma_a(t+)$ is at most countable.
\newline
Step 3. We prove that for $\pi^-_l$-a.e. $\mathcal Z \in ( \Gamma\times (0,1) \times B_R \times [0,M])^2$ there exists $\bar x \in \Q^2 \cap B_R$ such that 
\begin{equation*}
\mathcal Z \in \mathcal G^>_{\bar x,l} \times \mathcal G^<_{\bar x,l}.
\end{equation*}
Let us consider $\delta>0$ from Step 2. From Property (1) and \eqref{E_characteristic} it follows that for every $s \in (t-\delta,t)$ it holds
\begin{equation}\label{E_gamma}
 \gamma_x(s)\cdot e_l > \gamma_x(t)\cdot e_l - ie^{ia}\cdot e_l (\gamma_x(t)\cdot e_l^\perp- \gamma_x(s)\cdot e_l^\perp)
\end{equation}
and similarly for every $s'\in (t'-\delta,t')$
\begin{equation}\label{E_gamma'}
 \gamma'_x(s')\cdot e_l < \gamma'_x(t')\cdot e_l - ie^{ia'}\cdot e_l (\gamma'_x(t')\cdot e_l^\perp- \gamma'_x(s')\cdot e_l^\perp).
\end{equation}
Being $\pi^-$ concentrated on $\G$, for $\pi^-_l$-a.e. $\mathcal Z \in ( \Gamma\times (0,1) \times B_R \times [0,M])^2$ it also holds
$a=a'$ and $\gamma_x(t)=x=\gamma'_x(t')$. 
Let us consider 
\begin{equation*}
y\in \left(\gamma_x(t)\cdot e_l^\perp-\frac{\delta}{100}, \gamma_x(t)\cdot e_l^\perp\right) \cap \sqrt 2\Q.
\end{equation*}
Then there exist $s\in (t-\delta,t)$ and $s'\in (t'-\delta,t')$ such that $\gamma_x(s)\cdot e_l^\perp = y = \gamma'_x(s')\cdot e_l^\perp$.
It follows from \eqref{E_gamma} and \eqref{E_gamma'} that 
\begin{equation*}
\begin{split}
\gamma'_x(s')\cdot e_l < &~ \gamma'_x(t')\cdot e_l - ie^{ia'}\cdot e_l (\gamma'_x(t')\cdot e_l^\perp- \gamma'_x(s')\cdot e_l^\perp) \\
=&~ x \cdot e_l - ie^{ia'}\cdot e_l (x\cdot e_l^\perp - y) \\
= &~ \gamma_x(t)\cdot e_l - ie^{ia}\cdot e_l (\gamma_x(t)\cdot e_l^\perp- \gamma_x(s)\cdot e_l^\perp) \\
< &~ \gamma_x(s)\cdot e_l.
\end{split}
\end{equation*}
Let $z \in (\gamma'_x(s')\cdot e_l ,  \gamma_x(s)\cdot e_l) \cap \sqrt 2 \Q$ and set $\bar x = ze_l + y e_l^\perp$.
By construction it holds
\begin{equation*}
\mathcal Z \in \mathcal G^>_{\bar x,l} \times \mathcal G^<_{\bar x,l}
\end{equation*}
and since $e_l, e_l^\perp \in (\sqrt 2 \Q)^2$, then $\bar x \in \Q^2$.
\newline
Step 4. It follows by Step 3 that
\begin{equation}\label{E_inclusion}
\pi^-_l \le \pi^-_\llcorner \left( \bigcup_{\bar x \in \Q^2 \cap B_R}  \mathcal G^>_{\bar x,l} \times \mathcal G^<_{\bar x,l} \right).
\end{equation}
Since by Step 1 we have that $(p^1_x)_\sharp \pi^-_{\bar x,l}$ is concentrated on $C_{f_{\bar x,l}}$, then the statement of the proposition follows from \eqref{E_inclusion}.
\end{proof}

\subsection{Rectifiability of $\nu^-_{\mathrm{jump}}$}\label{Ss_nu_j}
In the next lemma we prove a regularity density estimate at a point $\bar x$ provided that the entropy dissipation measure decays faster than in a shock point. 
\begin{lemma}\label{L_density}
Let $(\gamma, t^-_\gamma,t^+_\gamma) \in  \Gamma_h$, $\bar t \in (t^-_\gamma,t^+_\gamma)$ and set $\bar x = \gamma_x(\bar t)$ and $\bar a = \gamma_a(\bar t-) \vee \gamma_a(\bar t+)$. Then there exists an absolute constant $c>0$ such that for every $\delta \in (0,\pi/2)$ at least one of the following holds true:
\begin{enumerate}
\item \begin{equation*}
\liminf_{r\to 0} \frac{\{x \in B_r(\bar x): \phi(x) \ge \bar a -\delta\}}{r^2}\ge c\delta;
\end{equation*}
\item
\begin{equation}\label{E_diss_shock}
\liminf_{r\to 0}\frac{\nu(B_r(\bar x))}{r}\ge c \delta^3.
\end{equation}
\end{enumerate}
\end{lemma}
\begin{proof}
We assume without loss of generality that $\bar a = \gamma_a(\bar t-)$ and we let $\delta_1>0$ be such that for every 
$t \in (\bar t - \delta_1, \bar t)$ it holds $\gamma_a(t) \in (\bar a -\delta/5, \bar a + \delta/5)$. We moreover set $\bar r = \delta_1/2$
so that for every $r \in (0,\bar r)$ there exists $t_r \in (\bar t - \delta_1,\bar t)$ such that $\gamma_x(t_r) \in \partial B_r(\bar x)$ and
$\gamma_x(t)\in B_r(\bar x)$ for every $t\in (t_r,\bar t)$.
Since $\gamma \in  \Gamma_h$ and $\gamma_a (t)\ge \bar a - \delta/5$ for every $t \in (t_r,\bar t)$, then there exists $\e>0$ such that 
\begin{equation*}
\L^2(\{ x \in S_{\e,\bar r} : \phi(x) \ge \bar a -\delta/5\}) \ge \e \bar r, \qquad \mbox{where} \qquad S_{\e,\bar r}:=\gamma_x((t_r,\bar t)) + B_\e(0).
\end{equation*}

For every $(\gamma,t^-_\gamma,t^+_\gamma) \in  \Gamma$ we consider the nontrivial interiors
$(t^-_{\gamma,i},t^+_{\gamma,i})_{i=1}^{N_{\gamma}}$ of the  connected components of $\gamma_a^{-1}((\bar a - \delta), \bar a - \frac{2}{5}\delta)$ which intersect $\gamma^{-1}(S_{\e,r}\times (\bar a - \frac{4}{5}\delta, \bar a - \frac{3}{5}\delta))$.
Notice that we have the estimate
\begin{equation*}
N_\gamma \le 1 + \frac{5}{\delta}\TV \gamma_a.
\end{equation*}

For every $i\in \N$ we consider
\begin{equation*}
 \Gamma_i:= \{(\gamma, t^-_\gamma,t^+_\gamma)\in  \Gamma : N_{\gamma}\ge i\}
\end{equation*}
and the measurable restriction map
\begin{equation*}
\begin{split}
R_{i}:  \Gamma_{i} & \to  \Gamma \\
(\gamma,t^-_\gamma,t^+_\gamma) & \mapsto (\gamma, t^-_{\gamma,i},t^+_{\gamma,i})
\end{split}
\end{equation*}
We finally consider the measure 
\begin{equation*}
\tilde \omega_{h}:= \sum_{i=1}^\infty (R_{i})_\sharp \left(\omega_h\llcorner  \Gamma_i\right).
\end{equation*}
Notice that $\tilde \omega_{h}\in \M_+( \Gamma)$ since for every $N>0$
\begin{equation*}
\left\| \sum_{i=1}^N (R_{i})_\sharp \left(\omega_h\llcorner  \Gamma_i\right) \right\| \le \int_{ \Gamma} N_{\gamma} d\omega_h \le  \int_{ \Gamma} \left(1 + \frac{5}{\delta}\TV \gamma_a\right) d\omega_h(\gamma) < \infty.
\end{equation*}
The advantage of using the restrictions introduced above is in the following estimate:
by an elementary transversality argument there exists an absolute constant $\tilde c>0$ such that for $\tilde \omega_h$-a.e. 
$(\gamma,t^-_\gamma,t^+_\gamma) \in  \Gamma$ it holds
\begin{equation}\label{E_trans}
\L^1\left(\left\{t \in (t^-_\gamma,t^+_\gamma) \in  \Gamma: \gamma(t) \in S_{\e,r}\times \left(\bar a - \frac{4}{5}\delta, \bar a - \frac{3}{5}\delta\right)\right\} \right) \le \tilde c \frac{\e}{\delta}.
\end{equation}
By construction we have that for every $t\in (0,1)$ it holds 
\begin{equation}\label{E_trans2}
( e_t)_\sharp \tilde \omega_h \ge \L^3 \llcorner \left\{(x,a) \in S_{\e,r}\times \left(\bar a - \frac{4}{5}\delta, \bar a - \frac{3}{5}\delta\right): \phi(x) \ge a \right\}.
\end{equation}
Since the measure of this set is at least $\e r \delta/5$, then it follows by \eqref{E_trans} and \eqref{E_trans2} that   
\begin{equation}\label{E_lower}
\tilde \omega_h ( \Gamma) \ge  \e r \frac{\delta}{5}\cdot \frac{\delta}{\tilde c \e}= \frac{\delta^2}{5 \tilde c}r.
\end{equation}
We consider $ \Gamma =  \Gamma_1 \cup  \Gamma_2$, where 
\begin{equation*}
  \Gamma_1 : = \{ (\gamma,t^-_\gamma,t^+_\gamma) \in  \Gamma : t^+_\gamma - t^-_\gamma \ge r\} \quad \mbox{and} \quad  
  \Gamma_2 : = \{ (\gamma,t^-_\gamma,t^+_\gamma) \in  \Gamma : t^+_\gamma - t^-_\gamma < r\}.
\end{equation*}

For $\tilde \omega_h$-a.e. $(\gamma,t^-_\gamma,t^+_\gamma) \in  \Gamma_1$ it holds
\begin{equation*}
\L^1\left( \left\{t \in (t^-_\gamma, t^+_\gamma): \gamma (t) \in B_{2r}(\bar x) \times \left(\bar a - \delta, \bar a - \frac{2}{5}\delta \right) \right\} \right) \ge r,
\end{equation*}
while for  $\tilde \omega_h$-a.e. $(\gamma,t^-_\gamma,t^+_\gamma) \in  \Gamma_2$ we have
\begin{equation*}
\gamma_x (t^-_\gamma,t^+_\gamma) \subset B_{2r}(\bar x) \qquad \mbox{and} \qquad \TV \gamma_a \ge \frac{2}{5}\delta.
\end{equation*}
It follows from \eqref{E_lower} that at least one of the following holds:
\begin{equation}\label{E_dicothomy}
\tilde \omega_h ( \Gamma_1) \ge  \frac{\delta^2}{10 \tilde c}r \qquad \mbox{or} \qquad \tilde \omega_h ( \Gamma_2) \ge  \frac{\delta^2}{10 \tilde c}r.
\end{equation}
If the second condition holds then we have that 
\begin{equation*}
\nu (B_{2r}(\bar x)) \ge |U_\phi|(B_{2r}(\bar x)\times (\bar a -\delta, \bar a)) \ge \frac{\delta^3}{25 \tilde c}r
\end{equation*}
so that the second condition in the statement is satisfied.
Otherwise we assume that the first condition in \eqref{E_dicothomy} holds:
since for every $t\in (0,1)$ 
\begin{equation*}
( e_t)_\sharp \tilde \omega_h \le \chi \L^3
\end{equation*}
it follows from \eqref{E_dicothomy} and Fubini theorem that 
\begin{equation*}
\L^2\left(\left\{x \in B_{2r}(\bar x): \phi (x)\ge \bar a -\delta\right\}\right) \ge \frac{\delta^2}{10 \tilde c}r \cdot \frac{5r}{3\delta} = \frac{\delta}{6\tilde c}r^2
\end{equation*}
so that the first condition in the statement holds true.
\end{proof}
\begin{remark}
We observe that the third power in  \eqref{E_diss_shock} is optimal; this is related to the fact that the optimal regularity of $\phi$ is 
$B^{1/3,3}_{\infty,loc}(\Omega)$, see \cite{GL_eikonal}.
\end{remark}

We also state the same result for curves in $ \Gamma_e$, whose proof is analogous to the one of Lemma \ref{L_density}.
\begin{lemma}\label{L_density2}
Let $(\gamma, t^-_\gamma,t^+_\gamma) \in  \Gamma_e$, $t \in (t^-_\gamma,t^+_\gamma)$ and set $\bar x = \gamma_x(t)$ and $\bar a = \gamma_a(t-) \wedge \gamma_a(t+)$. Then there exists an absolute constant $c>0$ such that for every $\delta \in (0,\pi/2)$ at least one of the following holds true:
\begin{enumerate}
\item \begin{equation*}
\liminf_{r\to 0} \frac{\{x \in B_r(\bar x): \phi(x) \le \bar a + \delta\}}{r^2}\ge c\delta;
\end{equation*}
\item
\begin{equation*}
\liminf_{r\to 0}\frac{\nu(B_r(\bar x))}{r}\ge c \delta^3.
\end{equation*}
\end{enumerate}
\end{lemma}

The main result of this section is the following:
\begin{proposition}\label{P_j}
For $\nu^-_{\mathrm{jump}}$-a.e. $x \in B_R$ 
\begin{equation}\label{E_beinJ}
\limsup_{r\to 0} \frac{\nu (B_r(x))}{r}>0.
\end{equation}
\end{proposition}
\begin{proof}
For $\nu^-_{\mathrm{jump}}$-a.e. $\bar x \in B_R$ one of the following holds:
\begin{enumerate}
\item there exist $(\gamma,t^-_\gamma,t^+_\gamma, t,x,a) \in \G^-_{h, \mathrm{jump}}$ and  $(\gamma',{t^-_\gamma}',{t^+_\gamma}', t',x',a') \in  \Gamma_e$ such that 
\begin{equation*}
x= x'= \bar x \qquad  \mbox{and} \qquad \gamma'_a(t'+)\le a'=a < \gamma_a(t-).
\end{equation*}
\item there exist $(\gamma',{t^-_\gamma}',{t^+_\gamma}', t',x',a') \in \G^+_{e, \mathrm{jump}}$ and $(\gamma,t^-_\gamma,t^+_\gamma, t,x,a) \in  \Gamma_e$ such that 
\begin{equation*}
x= x'= \bar x \qquad  \mbox{and} \qquad \gamma'_a(t'+)< a'=a \le \gamma_a(t-).
\end{equation*}
\end{enumerate}
Being the two cases equivalent we consider only the first one.
We apply Lemma \ref{L_density} to the curve $\gamma$ and Lemma \ref{L_density2} to the curve $\gamma'$ with 
$\delta =(\gamma_a(t-) - a )/3$. 
If condition (2) holds in at least one of the two cases then the statement follows, otherwise both the following inequalities are satisfied:
\begin{equation*}
\liminf_{r\to 0} \frac{\{x \in B_r(\bar x): \phi(x) \ge \gamma_a(t-) - \delta\}}{r^2}\ge c\delta^2, \qquad 
\liminf_{r\to 0} \frac{\{x \in B_r(\bar x): \phi(x) \le \gamma_a(t-) - 2\delta\}}{r^2}\ge c\delta^2.
\end{equation*}
This condition excludes that $\bar x$ is a point of vanishing mean oscillation of $\phi$, therefore $\bar x \in \Sigma$ by Theorem \ref{T_intro}, namely \eqref{E_beinJ} holds true.
\end{proof}

\subsection{Conclusion}
Collecting the results in Sections \ref{Ss_nu_l} and \ref{Ss_nu_j} we obtain the rectifiability of the measure $(p_x)_\sharp U^-_\phi$.
\begin{proposition}
The measure $(p_x)_\sharp U^-_\phi$ is 1-rectifiable.
\end{proposition}
\begin{proof}
We first observe that since $\pi^-$ is concentrated on $\G$ and 
\begin{equation*}
\G \subset  \big(\G^-_{h,\mathrm{jump}} \times \G^+_e\big) \cup \big(\G^-_h \times \G^+_{e,\mathrm{jump}}\big) \cup 
\left( \bigcup_{l=0}^L  \left(  \G^-_{h,l} \times \G^+_{e,l} \right) \right),
\end{equation*}
then it follows from the definitions of $\pi^-_l$ and $\pi^-_{\mathrm{jump}}$ that 
\begin{equation*}
\pi^- \le \pi^-_{\mathrm{jump}}+ \sum_{l=0}^L \pi^-_l.
\end{equation*}
In particular
\begin{equation*}
(p_x)_\sharp U^-_\phi = (p^1_x)_\sharp \pi^- \le (p^1_x)_\sharp  \pi^-_{\mathrm{jump}} +  \sum_{l=0}^L  (p^1_x)_\sharp\pi^-_l.
\end{equation*}
Since $(p^1_x)_\sharp \pi^-_l$ is 1-rectifiable for every $l=0,\ldots,L$ by Proposition \ref{P_l} and $(p^1_x)_\sharp  \pi^-_{\mathrm{jump}}$ is 1-rectifiable by Proposition \ref{P_j} and Theorem \ref{T_intro}, then also $(p_x)_\sharp U^-_\phi$ is 1-rectifiable.
\end{proof}

As mentioned at the beginning of this section, the rectifiability of the positive part $(p_x)_\sharp U^+_\phi$ can be proven following the same 
procedure. Therefore this concludes the proof of Theorem \ref{T_main}.

%
%
%
%
%
%

\end{document}